\documentclass[11pt]{amsart}

\usepackage[all]{xy}
\usepackage{fullpage}
\usepackage{latexsym}
\usepackage{amsmath}
\usepackage{amsfonts}
\usepackage{amssymb}
\usepackage{amsthm}
\usepackage{eucal}
\usepackage{enumerate,yfonts}
\usepackage{mathrsfs}
\usepackage{graphicx}
\usepackage{graphics}
\usepackage{epstopdf}
\usepackage{amscd}
\usepackage{bbm}
\usepackage{hyperref}
\usepackage{url}
\usepackage{color}
\usepackage{bbm}
\usepackage{cancel}
\usepackage{enumerate}
\usepackage{amsmath,amsthm}
\usepackage{amssymb}
\usepackage{epsfig}
\usepackage{pstricks}
\usepackage{xy}
\usepackage{xypic}

\newtheorem{thm}{Theorem}[section]
\newtheorem{corollary}[thm]{Corollary}
\newtheorem{lemma}[thm]{Lemma}

\newtheorem{proposition}[thm]{Proposition}
\newtheorem{prop}[thm]{Proposition}
\newtheorem{conjecture}[thm]{Conjecture}
\newtheorem{thm-dfn}[thm]{Theorem-Definition}

\theoremstyle{definition}
\newtheorem{definition}[thm]{Definition}

\numberwithin{equation}{section}

\theoremstyle{remark}
\newtheorem{remark}{Remark}[section]
\newtheorem{example}[remark]{Example}

\newcommand{\fg}{{\mathfrak g}}
\newcommand{\ft}{{\mathfrak t}}

\newcommand{\rW}{{\mathrm W}}

\newcommand{\bC}{{\mathbb C}}
\newcommand{\bX}{{\mathbb X}}
\newcommand{\bG}{{\mathbb G}}
\newcommand{\bZ}{{\mathbb Z}}

\newcommand{\mS}{\mathcal{S}}

\newcommand{\mE}{\mathcal{E}}
\newcommand{\mF}{\mathcal{F}}

\newcommand{\mM}{\mathcal{M}}
\newcommand{\mT}{\mathcal{T}}
\newcommand{\mO}{\mathcal{O}}
\newcommand{\mL}{\mathcal{L}}

\newcommand{\mG}{\mathcal{G}}

\newcommand{\sH}{\mathscr{H}}

\newcommand{\sD}{\mathscr{D}}

\newcommand{\on}{\operatorname}

\newcommand{\ra}{\rightarrow}

\newcommand{\is}{\simeq}

\newcommand{\Loc}{\on{Loc}}

\newcommand{\quash}[1]{}  
\newcommand{\nc}{\newcommand}

\newcommand{\bbA}{{\mathbb A}}

\newcommand{\bbC}{{\mathbb C}}
\newcommand{\bbD}{{\mathbb D}}

\newcommand{\bbQ}{{\mathbb Q}}

\newcommand{\bbX}{{\mathbb X}}

\newcommand{\bbZ}{{\mathbb Z}}

\newcommand{\calC}{{\mathcal C}}

\newcommand{\calF}{{\mathcal F}}

\newcommand{\calM}{{\mathcal M}}

\newcommand{\calR}{{\mathcal R}}
\newcommand{\calS}{{\mathcal S}}
\newcommand{\calT}{{\mathcal T}}

\newcommand{\calX}{{\mathcal X}}
\newcommand{\calY}{{\mathcal Y}}
\newcommand{\calZ}{{\mathcal Z}}

\nc{\al}{{\alpha}} \nc{\be}{{\beta}} \nc{\ga}{{\gamma}}
\nc{\ve}{{\varepsilon}} \nc{\Ga}{{\Gamma}} 
\nc{\La}{{\Lambda}}

\newcommand{\Dmod}{\on{D-mod}}

\nc{\ad }{{\on{ad }}}

\nc{\aff}{{\on{aff}}} \nc{\Aff}{{\mathbf{Aff}}}

\nc{\der}{{\on{der}}}

\nc{\diag}{{\on{diag}}}

\nc{\Fl}{{\calF\ell}}

\nc{\Hg}{{\on{Higgs}}}
\newcommand{\Hom}{{\on{Hom}}}
\newcommand{\id}{{\on{id}}}
\nc{\Id}{{\on{Id}}}

\nc{\Ind}{{\on{Ind}}}
\newcommand{\Lie}{{\on{Lie}}}
\nc{\Op}{{\on{Op}}}

\newcommand{\pr}{{\on{pr}}}
\newcommand{\Res}{{\on{Res}}}
\nc{\res}{{\on{res}}}

\newcommand{\Spec}{{\on{Spec}}}

\nc{\tr}{{\on{tr}}}

\newcommand{\GL}{{\on{GL}}}
\nc{\GSp}{{\on{GSp}}} \nc{\GU}{{\on{GU}}} \nc{\SL}{{\on{SL}}}
\nc{\SU}{{\on{SU}}} \nc{\SO}{{\on{SO}}}

\nc{\oH}{\on{H}}
\nc{\Shv}{\on{Shv}}
\nc{\tShv}{\on{Shv}^\heart}
\nc{\nh}{{\Loc_{J^p}(\tau')}}
\nc{\bnh}{{\Loc_{\breve J^p}(\tau')}}

\nc{\bU}{{\overline{U}}} \nc{\IC}{{\on{IC}}}

\newcommand{\barQ}{\overline{\mathbb Q}_\ell}

\newcommand{\beqn}{\begin{equation*}}
\newcommand{\eeqn}{\end{equation*}}

\newcommand{\beq}{\begin{equation}}
\newcommand{\eeq}{\end{equation}}

\newcommand{\sign}{\on{sign}}

\newcommand{\ul}{{\underline\lambda}}

\nc{\ot}{\otimes}

\newcommand{\rS}{{\mathrm S}}

\quash{
\topmargin-0.5cm \textheight22cm \oddsidemargin1.2cm \textwidth14cm}
\begin{document}
\title{A vanishing conjecture: the $\GL_n$ case}
\author{}
        \address{}
        \email{}
         \author{Tsao-Hsien Chen}
        \address{School of Mathematics,
      University of Minnesota, Twin Cities}
         \email{chenth@umn.edu}
\thanks{}
\thanks{}

\begin{abstract}
In this article we propose a vanishing conjecture for a certain class of $\ell$-adic complexes 
on a reductive group $G$ which can be regraded as a generalization of the acyclicity of 
the Artin-Schreier sheaf.
We show that the vanishing conjecture
contains, as a special case, a conjecture of Braverman and Kazhdan 
on the acyclicity of 
$\rho$-Bessel sheaves~\cite{BK1}.
Along the way, we introduce a certain class of 
Weyl group equivariant
$\ell$-adic complexes on a maximal torus called \emph{central complexes}
and relate the category of central complexes to
the Whittaker category on $G$.
We prove the vanishing conjecture in the case when $G=\GL_n$.

\end{abstract}

\maketitle     
\setcounter{tocdepth}{1} 
\tableofcontents

\section{Introduction}
\subsection{}
The Artin-Schreier sheaf $\mL_\psi$ on the additive group $\bG_a$ over
an algebraic closure of a finite field
has the following basic and important 
cohomology vanishing property
\[
\oH^*_c(\bG_a,\mL_\psi)=0.
\]
If we identify $\bG_a$ as the unipotent radical $U$ of the standard Borel $B$ in $\SL_2$, 
then the acyclicity of the Artin-Schreier sheaf above can be restated as follows.
Let $\tr:\SL_2\to\mathbb \bG_a$ be the trace map
and consider the pull back $\Phi=\tr^*\mL_\psi$.
For any $x\in \SL_2\setminus B$,  
we have the following cohomology vanishing property
\beq\label{vanishing sl_2}
\oH^*_c(Ux,i^*\Phi)=0.
\eeq
Here $i:Ux\to \SL_2$ is the natural inclusion map.
Indeed, 
it follows from the fact that 
the trace map $\tr$ restricts to an isomorphism 
$Ux\is\mathbb \bG_a$ for $x\in \SL_2\setminus B$.

In this article we 
propose a vanishing conjecture for a certain class of $\ell$-adic complexes 
on a reductive group $G$ generalizing the 
cohomology vanishing in~\eqref{vanishing sl_2}, and
hence can be regraded as a generalization of
the acyclicity of the Artin-Schreier sheaf.

We show that the vanishing conjecture implies a conjecture of Braverman and Kazhdan
on the acyclicity of $\rho$-Bessel sheaves \cite{BK1}
(Theorem \ref{main 1} and Corollary \ref{vanishing implies BK}) and 
we prove the vanishing conjecture in the case 
$G=\GL_n$ (Theorem \ref{main 2}).
Along the way, we introduce a certain class of 
Weyl group equivariant
$\ell$-adic complexes on a maximal torus called \emph{central complexes}
and relate the category of central complexes to the 
Whittaker category on $G$ (or rather, the de Rham counterpart of the Whittaker category). 

The proof of the vanishing conjecture for $\GL_n$
generalizes the one in \cite{CN}
for the proof of Braverman-Kazhdan
conjecture for $\GL_n$ using mirabolic subgroups.
A new ingredient here is 
a generalization of Deligne's result of 
symmetric group actions on Kloosterman sheaves 
to the setting of central complexes
(Proposition \ref{Deligne's lemma}).

We now describe the paper in more details.
\subsection{Central complexes}
Let $k$ be an algebraic closure of a finite field 
$k_0$ with $q$-element of characteristic $p>0$.
We fix a prime number $\ell$ different from $p$. 
Let $G$ be a connected reductive group over 
$k$.
Let $T$ be a maximal tour of $G$ and 
$B$ be a Borel subgroup containing $T$ with 
unipotent radical $U$. 
Denote by
$\rW=T\backslash N_G(T)$ the Weyl group, where $N_G(T)$
is the normalizer of $T$ in $G$.
Let 
$\calC(T)(\barQ)$ be the set consisting of 
characters of the tame \'etale fundamental group $\pi_1(T)^t$
(see Section \ref{C(T)}).
For any $\chi\in\calC(T)(\barQ)$ we denote by 
$\mL_\chi$ the corresponding tame local system on
$T$ (a.k.a the Kummer local system associated to $\chi$).
The Weyl group $\rW$ acts naturally on $\calC(T)(\barQ)$ and for any $\chi\in\calC(T)(\barQ)$
we denote by $\rW_\chi'$ the stabilizer of 
$\chi$ in $\rW$ and 
$\rW_\chi\subset\rW'_\chi$, the subgroup of $\rW_\chi'$ generated by those reflections 
$s_\alpha$ such that the pull-back 
$(\check\alpha)^*\mL_\chi$ is isomorphic to the trivial local system, 
 where 
$\check\alpha:\bG_m\to T$ is the coroot associated 
to $\alpha$. The group
$\rW_\chi$ is a normal subgroup of $\rW_\chi'$ and,
in general, we have $\rW_\chi\subsetneq\rW_\chi'$ (see Example \ref{W_chi}).\footnote{The group $\rW_\chi$ plays an important role in the study of 
representations of finite reductive groups and character sheaves (see, e.g., \cite{Lu}).}

Denote by
$\sD_\rW(T)$ 
the $\rW$-equivariant derived category of constructible 
$\ell$-adic complexes on $T$.
For any $\mF\in\sD_\rW(T)$
and $\chi\in\calC(T)(\barQ)$, the 
$\rW$-equivariant structure on $\mF$
together with
the natural $\rW_\chi'$-equivivariant structure on $\mL_\chi$
give rise to an
action of 
$\rW'_\chi$ on the
\'etale cohomology groups
$\oH_c^*(T,\mF\otimes\mL_\chi)$ (resp. $\oH^*(T,\mF\otimes\mL_\chi)$).
In particular, we get an action of the subgroup 
$\rW_\chi\subset\rW_\chi'$ on the cohomology groups above.
Denote by $\sign_\rW:\rW\to\{\pm1\}$ the sign character of 
$\rW$.

We have the following key definition.
\begin{definition}\label{def of central}
A $\rW$-equivariant complex $\mF\in\sD_\rW(T)$ 
is called \emph{central} (resp. $*$-\emph{central}) if for any tame character $\chi\in\calC(T)(\barQ)$, the group
$\rW_\chi$ acts on \[\oH_c^*(T,\mF\otimes\mL_\chi)\ \ \ \ \ 
(\text{resp.}\ \ \oH^*(T,\mF\otimes\mL_\chi))\]
via the sign character $\sign_\rW$.

\end{definition}

\begin{remark}
The Verdier duality maps central objects to $*$-central objects
and vice versa.
\end{remark}

\begin{example}\label{sl_2 example}
Consider the case $G=\SL_2$.
Let $\tr_T:T\is\bG_m\to\bG_a, t\to t+t^{-1}$ and 
consider $\mF=\tr_T^*\mL_\psi[1]$ with the canonical $\rW$-equivaraint structure.
We claim that $\mF$ is central.
For this we observe that
$\rW_\chi\neq e$ if and only if $\chi$ is trivial.
Thus $\mF$ is central if and only if
$\rW$ acts on
$\oH_c^*(T,\mF)$ by the sign character, equivalently,
the non-trivial involution $\sigma\in\rW$ acts by $-1$ on $\oH_c^*(T,\mF)$.
On the other hand, we have $\oH_c^*(T,\mF)^\sigma\is\oH_c^*(\bG_a,\mL_\psi\otimes(\tr_{T,!}\mathbb Q_\ell)^\sigma)\is
\oH_c^*(\bG_a,\mL_\psi)=0$
and it implies $\mF$ is central.

\quash{
Let $G=\SL_2$ and $\tr_T:T\to\bG_a, t\to t+t^{-1}$
be the trace map.
The pull back $\mF=\tr_T^*\mL_\psi$ is naturally $\rW$-equivariant
and we claim  that 
$\mF$
 is both central and $*$-central.
 Indeed, in this case  
$\rW_\chi$ is nontrivial if and only if $\chi$ is trivial, and 
in that case a simple computation shows that the
$\rW$-action on $\oH_c^*(T,\mF)$ (resp. $\oH^*(T,\mF)$) is given by the sign character.}
\end{example}

\begin{example}
Using Mellin transforms in \cite{GL}, one can associate to each
$\rW$-orbit $\theta$ in $\calC(T)(\barQ)$ a 
tame central local system on $T$ (see Section \ref{tame central}).

\end{example}

\quash{
\begin{remark}
The category of central complexes (resp. $*$-central complexes)
admits a symmetric monoidal structure (see Section \ref{Mellin}).

\end{remark}}

\subsection{Statement of the vanishing conjecture}\label{statement of vanishing conj}
We have the induction functor \[\Ind_{T\subset B}^G:\sD(T)\to \sD(G)\]
between the 
derived categories of $\ell$-adic sheaves on $T$ and $G$.
For $\mF\in\sD_\rW(T)$, the $\rW$-equivariant 
structure on $\mF$ defines a $\rW$-action on
$\Ind_{T\subset B}^G(\mF)$ and we denote by
\[\Phi_\mF:=\Ind_{T\subset B}^G(\mF)^\rW\]
the $\rW$-invariant factor in $\Ind_{T\subset B}^G(\mF)$
(see Section \ref{W action}). 
We propose the the following conjecture on
acyclicity of $\Phi_{\mF}$ over certain affine subspaces of $G$, called the 
vanishing conjecture:
\begin{conjecture}\label{vanishing conj}
Assume  $\mF\in\sD_\rW(T)$ is central (resp. $*$-central).
For any $x\in G\setminus B$, 
we have the following cohomology vanishing
\beq\label{statement}
\ \ \ \ \ \ \ \ \ \ \ \oH_c^*(Ux,i^*\Phi_{\mF})=0\ \ \ \ \ \ (\text{resp.}\ \oH^*(Ux,i^!\Phi_{\mF})=0)
\eeq
where $i:Ux\to G$ is the natural inclusion map.
Equivalently, 
the complex $\pi_!(\Phi_\mF)$ (resp. $\pi_*(\Phi_\mF)$) is supported on the closed subset 
$T=U\backslash B\subset U\backslash G$.
Here $\pi:G\to U\backslash G$ is the quotient map.

\end{conjecture}

\quash{
Let $\pi:G\to U\backslash G$ be the quotient map.
It follows from proper base change that 
\eqref{statement} is equivalent to the statement that
the complex $\pi_!(\Phi_\mF)$ (resp. $\pi_*(\Phi_\mF)$) is supported on the closed subset 
$T=U\backslash B\subset U\backslash G$.
}

\begin{remark}\label{*=!}
Note that the Verdier duality $\bbD$ interchanges central objects with $*$-central objects
and $\oH_c^*(Ux,i^*\Phi_\mF)$ is dual to $\oH^*(Ux,i^!\bbD\Phi_\mF)$. Thus
the conjecture above for central objects implies the one for $*$-central objects
and vice versa.
\end{remark}

\begin{remark}
The acyclicity of Artin-Schreier sheaf is an essential ingredient in the proof 
that the $\ell$-adic Fourier transform is a $t$-exact equivalence of categories. This property of 
$\ell$-adic Fourier transform had found several applications in 
number theory and representation theory.
We expect the vanishing conjecture would also 
have applications in number theory and representation theory (see Section \ref{generalized BK conj} 
and \ref{Main results} below for applications in the
Braverman-Kazhdan conjecture).
\end{remark}
\begin{remark}\label{center}
Denote by $\sD_G(G)$ 
the $G$-conjugation equivariant derived category on $G$
and by $\sD_U(U\backslash G)$, the Hecke category of 
$U$-equivariant derived category on $U\backslash G$. 
Let $\pi:G\to U\backslash G$ be the quotient map.
The push-forward $\pi_!$ induces a functor 
\[\pi_!:\sD_G(G)\to \sD_U(U\backslash G),\] 
and it is known that 
for any $\mM\in\sD_G(G)$,
the image of $\pi_!(\mM)\in\sD_U(U\backslash G)$ carries a canonical central structure, that is, we have a canonical isomorphism 
$\pi_!(\mM)*_!\mG\is\mG*_!\pi_!(\mM)$ for any $\mG\in \sD_U(U\backslash G)$ (here  $*_!$ is the 
convolution product on the Hecke category with respect to 
shriek push-forward).
It follows from Conjecture \ref{vanishing conj}
that 
$\pi_!(\Phi_{\mF})\is\mF$ for any central complex $\mF$.
In particular, it implies that any central complex $\mF$ 
carries a canonical central structure. This explains
the origin of the name ''central complexes``.

\end{remark}

\begin{example}
Let $\mF=\tr_T^*\mL_\psi[1]$ be as in Example \ref{sl_2 example}.
We claim that $\Phi_\mF\is\tr^*\mL_\psi[\dim\SL_2]$.
Indeed, both complexes are isomorphic to the IC-extensions of their restrictions to 
the regular semi-simple locus
$\SL_2^{\on{rs}}$, and using the fact that Grothendieck-Springer simultaneous resolution~\eqref{G-S resolution}
is a Cartesian over $\SL_2^{\on{rs}}$, it is easy to show that 
$\Phi_\mF|_{\SL_2^{\on{rs}}}\is\tr^*\mL_\psi[\dim\SL_2]|_{\SL_2^{\on{rs}}}$.
Thus
the vanishing conjecture  becomes \eqref{vanishing sl_2}, which is exactly the 
acyclicity of Artin-Schreier sheaf.
\end{example}

\subsection{Braverman-Kazhdan conjecture}\label{generalized BK conj}
We recall a construction, due to Braverman and Kazhdan,
of $\rho$-Bessel sheaf $\Phi_{G,\rho}$ attached to
a $r$-dimensional 
complex representation $\rho$ of the complex dual group $\check G$.

Let $\rho:\check G\to\GL(V_\rho)$ be such a representation.
The restriction of $\rho$ to the maximal torus $\check T$ is diagonalizable and there exists 
a collection of weights 
\[\underline\lambda=\{\lambda_1,...,\lambda_r\}\subset\bX^\bullet(\check T):=\Hom(\check T,\bC^\times)\]
such that there is an eigenspace decomposition 
\[V_\rho=\bigoplus_{i=1}^r V_{\lambda_i}\]
of $V_\rho$, where $\check T$ acts on $V_{\lambda_i}$ via the character $\lambda_i$.
One can regard $\underline\lambda$ as collection of 
co-characters of $T$ using the the canonical isomorphism
$\bX^\bullet(\check T)\is\bX_\bullet(T)$,
and define
\[
\Phi_{T,\rho}=\pr_{\ul,!}\tr^*\mL_\psi[r]
\]
\[\Phi^*_{T,\rho}=\pr_{\ul,*}\tr^*\mL_\psi[r]
\]
where 
\[\pr_\ul:=\prod_{i=1}^r\lambda_i:\bG_m^r\longrightarrow T,\ \ \ \ \tr:\bG_m^r\longrightarrow\bG_a, (x_1,...,x_r)\ra\sum_{i=1}^r x_i.\]
It is shown in \cite{BK2}, using Deligne's result about 
symmetric group actions on hypergeometric sheaves (see Proposition \ref{sign}),
that both $\Phi_{T,\rho}$ and $\Phi^*_{T,\rho}$
carry natural 
$\rW$-equivariant structures and the resulting objects in 
$\sD_\rW(T)$, denote again by $\Phi_{T,\rho}$
and $\Phi^*_{T,\rho}$, are called
$\rho$-Bessel sheaves on $T$.\footnote{
In \cite{BK1,BK2}, the authors called
$\Phi_{T,\rho}$  $\gamma$-sheaves on $T$. However, based on the fact that 
the classical $\gamma$-function is the Mellin transform of the Bessel function, we follow \cite{Ngo} and use the 
term $\rho$-Bessel sheaves instead of $\gamma$-sheaves.}
The $\rho$-Bessel sheaves on $G$ attached to $\rho$, denoted by $\Phi_{G,\rho}$ and $\Phi^*_{G,\rho}$, are defined as 
\[\Phi_{G,\rho}=\Ind_{T\subset B}^G(\Phi_{T,\rho})^{\rW}\]
\[\Phi^*_{G,\rho}=\Ind_{T\subset B}^G(\Phi^*_{T,\rho})^{\rW}.\]

In \cite[Conjecture 9.12]{BK1},
Braverman-Kazhdan
proposed the following conjecture on acyclicity of $\rho$-Bessel sheaves over certain affine subspaces of $G$:
 \begin{conjecture}\label{BK conj}
Let $\Phi_{G,\rho}$ (resp. $\Phi^*_{G,\rho}$)
be the $\rho$-Bessel sheaf
attached to a representation 
$\rho:\check G\to\GL(V_\rho)$ of the dual group. 
Then for any $x\in G\setminus B$, 
we have the following cohomology vanishing
\beq
\ \ \ \ \ \ \ \ \ \ \oH_c^*(Ux,i^*\Phi_{G,\rho})=0\ \ \ \ (\text{resp.}\ \ 
\oH^*(Ux,i^!\Phi^*_{G,\rho})=0)
\eeq
where $i:Ux\to G$ is the natural inclusion map.
Equivalently, 
the complex $\pi_!(\Phi_{G,\rho})$ (resp. $\pi_*(\Phi_{G,\rho}^*)$) is supported on the closed subset 
$T=U\backslash B\subset U\backslash G$.
Here $\pi:G\to U\backslash G$ is the quotient map.

\end{conjecture}

\begin{remark}
In \emph{loc.cit.} Conjecture \ref{BK conj} was stated for 
those representations $\rho$ with 
$\sigma$-positive weights 
(see Section \ref{Hypergeometric sheaves} for the definition of 
$\sigma$-positive weights).
It is shown in \cite{CN,BK2} that, under this positivity assumption, 
the $\rho$-Bessel sheaves $\Phi_{G,\rho}$ and $\Phi^*_{G,\rho}$ (resp. $\Phi_{T,\rho}$ and $\Phi^*_{T,\rho}$) 
are in fact perverse sheaves and
we have 
$\Phi_{G,\rho}\is\Phi^*_{G,\rho}$ (resp. 
$\Phi_{T,\rho}\is\Phi^*_{T,\rho}$) (see Proposition \ref{cleanness of gamma}).
This is a generalization of Deligne's theorem on Kloosterman sheaves \cite{D}.
We will see below that 
the vanishing conjecture (Conjecture \ref{vanishing conj}) implies the Braverman-Kazhdan conjecture for $\rho$-Bessel sheaves attached to
arbitrary representation $\rho$ of the dual group $\check G$.
\end{remark}

\begin{remark}
For a motivated introduction to the Braverman-Kazhdan conjecture and 
its role in the Langlands program see \cite{BK1,BK2,Ngo}
\end{remark}

\subsection{Main results}\label{Main results}
The following is the first main result of the paper.
Let $\Phi_{T,\rho}$ (resp. $\Phi^*_{T,\rho}$)
be the $\rho$-Bessel sheaf on $T$ attached to a representation 
$\rho:\check G\to\GL(V_\rho)$
of the dual group.
\begin{thm}\label{main 1}
For any tame character 
$\chi$ of $T$, the stabilizer subgroup $\rW'_\chi$ acts on 
\[\ \ \ \ \ \ \oH_c^*(T,\Phi_{T,\rho}\otimes\mL_\chi)  \ \ \ \ (\text{resp.}\ \  \oH^*(T,\Phi^*_{T,\rho}\otimes\mL_\chi))\] via the sign character $\sign_\rW$. 
In particular, the 
$\rho$-Bessel sheaf $\Phi_{T,\rho}$ (resp. $\Phi^*_{T,\rho}$)
 is central (resp. $*$-central)

\end{thm}

\begin{corollary}\label{vanishing implies BK}
Conjecture \ref{vanishing conj} implies Conjecture \ref{BK conj}

\end{corollary}
\begin{remark}\label{new context}
Theorem \ref{main 1} and 
Corollary \ref{vanishing implies BK} put the Braverman-Kazhdan conjecture in a wider context: it is a special case of a more general vanishing conjecture 
whose formulation does not involve representations of the dual group.
\end{remark}

Here is the second main result of the paper.
\begin{thm}\label{main 2}
Conjecture \ref{vanishing conj} is true for $G=\GL_n$.
\end{thm}
Now Corollary \ref{vanishing implies BK}
immediately implies:
\begin{corollary}\label{BK for GL_n}
Conjecture \ref{BK conj} is true for $G=\GL_n$.
\end{corollary}

Conjecture \ref{BK conj} for $\GL_n$ was proved by Cheng and Ng\^o 
\cite[Theorem 2.4]{CN} under some assumption on $\rho$.
The proof of Conjecture \ref{vanishing conj} for $\GL_n$, and hence Conjecture \ref{BK conj} for $\GL_n$
and arbitrary $\rho$, 
follows the one in \cite{CN}
using mirabolic subgroups.
However, in order to deal with the more general case of 
central complexes, one needs 
a new ingredient: a generalization of Deligne's result of 
symmetric group action on hypergeometric sheaves 
to the setting of central complexes (see Proposition \ref{Deligne's lemma}).

The discussions in the previous sections
have obvious counterpart in the de Rham setting, where 
Artin-Schreier sheaf $\mL_\psi$ is replaced by the exponential $D$-module $e^x$
and 
$\ell$-adic sheaves are replaced by holonomic $D$-modules. 
The proof of the main results that will be given in this paper
is entirely geometric and hence can be also applied to the de Rham setting.

\subsection{Whittaker categories}
The de Rham counterpart of the category of $*$-central perverse sheaves on $T$, to be called the category of $*$-central $D$-modules on $T$, 
also appears in the recent works of 
Ginzburg and Lonergan \cite{G,L} on
Whittaker $D$-modules, nil Hecke algebras, and quantum Toda lattices.
To be specific, in~\emph{loc.cit.} the authors proved that 
the category of Whittaker $D$-modules on $G$, denoted by
$\on{Whit}(G)$, is equivalent to
a certain full subcategory of the category of $\rW$-equivariant 
$D$-modules on $T$. 
It turns out that, as proven in \cite[Theorem 1.8]{C}, the later full subcategory, and
hence the Whittaker category $\on{Whit}(G)$, is equivalence to the 
categroy of $*$-central $D$-modules on $T$.
It would be interesting to establish similar results in the $\ell$-adic setting
and use it to give a description of $\ell$-adic counterpart of the Whittaker category in terms of 
$\rW$-equivariant sheaves on the maximal torus $T$.\footnote{
An question raised by V.Drinfeld according to \cite[Section 1.5]{G}.}

The equivalence bewteen $\on{Whit}(G)$ and the category of $*$-central $D$-modules on $T$ 
gives rise to a 
functor  $\on{Whit}(G)\to\Dmod(G\backslash G)$, where 
$\Dmod(G\backslash G)$ is the category of $G$-conjugation equivariant 
$D$-modules on $G$. 
In \cite{BG}, Ben-Zvi and Gunningham gave another construction of a functor from 
$\on{Whit}(G)$ to $
\Dmod(G\backslash G)$ and they conjectured that 
the essential image of their functor satisfies the same acyclicity in~\eqref{statement}, see \cite[Conjecture 2.9 and 2.14]{BG}.
Our results might be useful for studying their conjecture.

\subsection{Recent developments}
In the recent work \cite{C}, 
we established the vanishing conjecture for general reductive group $G$ and hence,
by Corollary \ref{vanishing implies BK}, the Braverman-Kazhdan 
conjecture for almost all characteristics.
The proof given in~\emph{loc. cit.} is different form the one given here: 
it first established the vanishing conjecture in the de Rham setting using 
the theory of Harish-Chandra bimodules and character $D$-modules and 
then deduced the positive characteristic case using 
character sheaves in mixed-characteristic and a spreading out argument.
The proof makes use of Harish-Chandra bimodules and hence is algebraic.
It will be interesting to have a geometric proof of the vanishing conjecture for 
general reductive groups, like the proof given here for $\GL_n$,
which treats the case 
of various ground fields and sheaf theories uniformly.

In the recent work \cite{LL}, G. Laumon and E. Letellier established 
the function theoretical version
of
 Braverman-Kazhdan conjecture for all reductive groups $G$.

\subsection{Organization}
We briefly summarize here the main goals of each section.
In Section \ref{Notations} we collect standard notation
in algebraic geometry and $\ell$-adic sheaves.
In Section \ref{ind and res} we study induction and restriction functors 
for $\ell$-adic sheaves on reductive groups.
In Section \ref{central complexes} we give a 
characterization of central complexes using the $\ell$-adic Mellin transform 
and use it to
construct examples tame central (resp. $*$-central) local systems on $T$.
In Section \ref{gamma sheaves} we prove Theorem \ref{main 1}.
In Section \ref{GL_n case} we prove Theorem \ref{main 2}.
\\

{\bf Acknowledgement.}
The paper is inspired by the 
lectures given by 
Ginzburg and Ng\^o on their 
works~\cite{G} and~\cite{CN}.
I thank Roman Bezrukavnikov,
Tanmay Deshpande,
Victor Ginzburg, Sam Gunningham, Augustus Lonergan, David Nadler, Ng\^o Bao Ch\^au, Lue Pan, and Zhiwei Yun for useful discussions.
I am grateful for the support of NSF grant DMS-1702337.

\section{Notations}\label{Notations}
In this article $k$ will be an algebraic closure of a finite field 
$k_0$ with $q$-element of characteristic $p>0$.
We fix a prime number $\ell$ different from $p$. 

For an algebraic stack $\calX$ over $k$, we denote by
$\sD(\calX)=D_c^b(\calX,\barQ)$ the bounded derived category of 
constructible $\ell$-adic complexes on $\calX$.
For a representable morphism $f:\calX\to\calY$,
the six functors $f^*,f_*,f_!,g^!,\otimes,\underline\Hom$ 
are understood in the derived sense. 
For a $k$-scheme $X$, 
sometimes we will write 
$R\Gamma(X,\barQ)=f_*\barQ$ (resp. $R\Gamma_c(X,\barQ)=f_!\barQ$),
where $f:X\to\on{Spec}k$ is the structure map.

For an algebraic group $H$ over $k$ acting on a 
$k$-scheme $X$, we denote by $H\backslash X$, the corresponding quotient stack.
Consider the case when $H$ is a finite group.
Then the pull-back along the quotient map $X\to H\backslash X$ induces an equivalence 
between 
$\sD(H\backslash X)$ and the (naive) $H$-equivariant derived category 
of $\ell$-adic complexes 
on $X$, denoted by $\sD_H(X)$, whose objects consist of pair 
$(\mF,\phi)$, where $\mF\in\sD(X)$ and $\phi:a^*\mF\is\pr^*\mF$ is an isomorphism 
satisfying the usual compatibility conditions (here $a$ and $\pr$ are the action and projection map from $H\times X$ to $X$ respectively).\footnote{This holds in a more general situation when the neutral component of $H$ is unipotent.} 
We will call an object $(\mF,\phi)$ in $\sD_H(X)$ a $H$-equivariant complex and 
$\phi$ a $H$-equivariant structure on $\mF$. 
For simplicity, 
we will write $\mF=(\mF,\phi)$ for an object in $\sD_H(X)$.

The category $\sD(\calX)$ has a natural perverse $t$-structure 
and we denote by $\sD(\calX)^\heartsuit$ the corresponding heart and 
$^p\tau_{\leq n},^p\tau_{\geq n}$ the perverse truncation functors.
For any $\mF\in \sD(\calX)$, the $n$-the perverse cohomology sheaf is defined as 
$^p\mathscr H^n(\mF)={^p}\tau_{\geq n}{^p}\tau_{\leq n}(\mF)[n]$.

We will denote by $\bG_a$ the additive group over $k$
and $\bG_m$ the multiplicative group over $k$.
We will fix a non-trivial character $\psi:\mathbb F_q\to\barQ^\times$ 
and denote by $\mL_\psi$ the corresponding Artin-Schreier sheaf on $\bG_a$.

\section{Induction and restriction functors}\label{ind and res}

\subsection{}
Let $G$ be a connected reductive group over $k$.
Let $T$ be a maximal torus of $G$, $B$ a Borel subgroup 
containing $T$ with unipotent radical $U$.
We denote by $\rW=T\backslash N_G(T)$ the Weyl group of $G$,
where $N_G(T)$ is the normalizer of $T$ in $G$.
We denote by $\sign=\pr^*\sign_\rW\in\sD_\rW(T)$ the pull back 
of
the sign representation $\sign_\rW$ of $\rW$ (regarding as an object in $\sD_\rW(\on{pt})$) along 
the projection $\pr:T\to\on{pt}$. 
Throughout the paper, we assume $p$ does not divide the order of the Weyl group $\rW$.

\subsection{}
Recall the Grothendieck-Springer simultaneous resolution 
of the Steinberg map $c:G\to\rW\backslash\backslash T$:
\beq\label{G-S resolution}
\xymatrix{\widetilde G\ar[r]^{\tilde q}\ar[d]^{\tilde c\ \ }&T\ar[d]^q\\
G\ar[r]^c&\rW\backslash\backslash T}
\eeq
where $\widetilde G$ is the closed subvariety of $G\times G/B$ 
consisting of 
pairs $(g,xB)$ such that $x^{-1}gx\in B$.
The map $\tilde c$ is given by $(g,xB)\to g$, and the map 
$\tilde q$ is given by $(g,xB)\to x^{-1}gx\on{\ mod} U\in B/U=T$.  
The induction functor $\Ind_{T\subset B}^G:\sD(T)\to\sD(G)$
is given by \[\Ind_{T\subset B}^G(\mF)=\tilde c_!\tilde q^*(\mF)[\dim G-\dim T].\]
We have the following equivalent construction of 
$\Ind_{T\subset B}^G$. Consider the fiber product 
$S=G\times_{\rW\backslash\backslash T}T$. 
The diagram \eqref{G-S resolution} induces a map
\beq
h:\widetilde G\to S=G\times_{\rW\backslash\backslash T}T
\eeq
which is proper and small, and an isomorphism 
over $S^{rs}=G^{rs}\times_{\rW\backslash\backslash T}T^{rs}$.
It follows that 
\beq\label{IC(S)}
h_!\barQ[\dim G]\is j_{!*}{\barQ}[\dim G]:=\IC(S,\barQ)
\eeq
where $j:S^{rs}\to S$ 
is the open embedding and $j_{!*}$ is the intermediate extension 
functor.
We have 
\beq\label{Ind IC}
\Ind^G_{T\subset B}(\mF)\is (p_G)_!(p_T^*(\mF)\otimes\IC(S,\barQ))[-\dim T]
\eeq
where $p_T:S\to T$ and $p_G:S\to G$ are the natural projection map.

Let functor $\Ind_{T\subset B}^G$ admits a left adjoint
$\Res_{T\subset B}^G:\sD(G)\to \sD(T)$, called the restriction functor, and is given by
\[\Res_{T\subset B}^G(\mF)= (q_B)_!i_B^*(\mF) \]
where $i_B:B\to G$ is the natural inclusion and 
$q_B:B\to T=U\backslash B$ is the quotient map.
More generally, one could define 
$\Res_{L\subset P}^G: \sD(G)\to \sD(L)$, for any 
pair $(L,P)$ where $L$ is a Levi subgroup of a parabolic subgroup $P$ of $G$.

We have the following exactness properties of 
induction and restriction functors:
\begin{proposition}\cite[Theorem 5.4]{BY}\label{exactness}
(1) The functor $\Ind_{T\subset B}^G$ maps perverse sheaves on $T$ to 
perverse sheaves on $G$.
(2) The functor 
$\Res_{L\subset P}^G$ maps 
$G$-conjugation equivariant perverse sheaves on $G$ to 
$L$-conjugation equivariant perverse sheaves on $L$.
\end{proposition}

\begin{remark}
The proposition above generalizes well-known results of Lusztig's on 
exactness of induction and restriction functors for character sheaves \cite{Lu}.
\end{remark}
\subsection{$\rW$-action}

\begin{prop}\label{W action}
(1)
Let $\mF\in\sD(T)$. 
For every $w\in\rW$ one has
a canonical isomorphism 
\[\Ind_{T\subset B}^G(\mF)\is\Ind_{T\subset B}^G(w^*\mF).\]
(2)
Let $\mF\in\sD_\rW(T)$.
There is a natural $\rW$-action on 
$\Ind_{T\subset B}^G(\mF)\in\sD(G)$ and, for any  
irreducible representation 
$\xi:\rW\to\GL(V_\xi)$ of $\rW$, there is canonical direct summand 
\beq\label{summand}
\Ind_{T\subset B}^G(\mF)^{\rW,\xi}\in\sD(G)
\eeq of 
$\Ind_{T\subset B}^G(\mF)$ such that 
we have a $\rW$-equivariant decomposition in $\sD(G)$
\beq\label{decomposition}
\Ind_{T\subset B}^G(\mF)\is\bigoplus_{(\xi,V_\xi)\in\on{Irr}\rW} V_\xi\otimes\Ind_{T\subset B}^G(\mF)^{\rW,\xi}.
\eeq
(3)
Let $\mF\in\sD_\rW(T)^\heartsuit$ be a $\rW$-equivariant
perverse sheaf. Then 
the direct summand $\Ind_{T\subset B}^G(\mF)^{\rW,\xi}$ in~\eqref{decomposition} is also a perverse sheaf.

\end{prop}
\begin{proof}
We have a $\rW$-action on $S=G\times_{\rW\backslash\backslash T}T$ 
given by $(g,t)\to (g,wt)$, $w\in\rW$, and
the projection maps $p_G$ and $p_T$ 
form $S$ to $G$ and $T$ are $\rW$-equivariant (for $p_G$, the $\rW$-action
on $G$ is the trivial action).
Consider the following commutative diagram
\beq\label{quotient diagram}
\xymatrix{G\ar[d]^{\id}&S\ar[d]^{\pi'}\ar[l]_{p_G}\ar[r]^{p_T}&T\ar[d]^\pi\\
G&\rW\backslash S\ar[r]^{p_T'}\ar[l]_{p_G'}&\rW\backslash T}
\eeq
where $\pi', p_G',p_T'$ are the natural quotient maps.
Note that, 
since $S^{rs}$ is $\rW$-invariant, 
the IC-complex $\IC(S,\barQ)=j_{*!}\barQ[\dim G]$ (here $j:S^{rs}\to S$)
descends to the IC-complex $\IC(\rW\backslash S,\barQ)$
on $\rW\backslash S$, in particular, $\IC(S,\barQ)$ is
$\rW$-equivariant.
It follows from~\eqref{Ind IC}  that 
\beq\label{w-action}
\Ind_{T\subset B}^G(w^*\mF)\is
(p_G)_!(w^*p_T^*(\mF)\otimes\IC(S,\barQ))[-\dim T]
\is
\eeq
\[\is
(p_G)_!(w^*p_T^*(\mF)\otimes w^*\IC(S,\barQ))[-\dim T]\is(p_G)_!w^*\big(p_T^*(\mF)\otimes\IC(S,\barQ)\big)[-\dim T]\is\]
\[\is
(p_G)_!\big(p_T^*(\mF)\otimes\IC(S,\barQ)\big)[-\dim T]\is
\Ind_{T\subset B}^G(\mF).\]
Part (1) follows.

Let $\mF\in\sD_\rW(T)$.
Since $w^*\mF\is\mF$ for any $w\in\rW$, we have 
 a canonical isomorphism 
\[a_w:\Ind_{T\subset B}^G(\mF)\is
\Ind_{T\subset B}^G(w^*\mF)\stackrel{\eqref{w-action}}\is\Ind_{T\subset B}^G(\mF),\]
and the assignment $w\to a_w, w\in\rW$
defines a $\rW$-action on $\Ind_{T\subset B}^G(\mF)$.
To show~\eqref{decomposition},
we observe that 
\[\IC(S,\barQ)\is(\pi')^*\IC(\rW\backslash S,\barQ),\]
and it implies
\beq\label{proj formula 1}
(\pi')_!\IC(S,\barQ)\is
(\pi')_!(\pi')^*\IC(\rW\backslash S,\barQ)\is
\bigoplus_{(\xi,V_\xi)\in\text{Irr}\rW} V_\xi\otimes(\IC(\rW\backslash S,\barQ)\otimes V_{\xi,S}\big)
\eeq
where $V_{\xi,S}\in\sD(\rW\backslash S)$ is the pull back
of $(\xi,V_\xi)\in \on{Rep}\rW\is\sD(\rW\backslash\on{pt})^\heartsuit$ along the projection 
$\rW\backslash S\to\rW\backslash\on{pt}$.

Let $\mF'\in\sD(\rW\backslash T)$ be such that $\pi^*\mF'\is\mF$.
It follows that
\[
\Ind_{T\subset B}^G(\mF)\stackrel{}\is
(p_G)_!\big(p_T^*(\mF)\otimes\IC(S,\barQ)\big)[-\dim T]\is\]
\[\is
(p_G')_!(\pi')_!\big((\pi')^*(p_T')^*\mF'\otimes\IC(S,\barQ)\big)[-\dim T]\is
(p'_G)_!\big((p'_T)^*\mF'\otimes(\pi'_!)\IC(S,\barQ)\big)[-\dim T]\stackrel{\eqref{proj formula 1}}\is\]
\[\is\bigoplus_{(\xi,V_\xi)\in\on{Irr}\rW} V_\xi\otimes\Ind_{T\subset B}^G(\mF)^{\rW,\xi}
\]
where 
\beq
\Ind_{T\subset B}^G(\mF)^{\rW,\xi}:=(p'_G)_!\big((p'_T)^*\mF'\otimes \IC(\rW\backslash S,\barQ)\otimes V_{\xi,S}\big)[-\dim T].
\eeq
Part (2) follows.
Part (3) follows from Proposition \ref{exactness}.
\end{proof}

\begin{remark}
Part (1) of the proposition generalizes  \cite[Theorem 2.5(3)]{BK1}.
\end{remark}

\begin{definition}\label{invaraint factor}
For any $\rW$-equivariant complex $\mF\in\sD_\rW(T)$, we will write
\beq\label{W-inv summand}
\Ind_{T\subset B}^G(\mF)^\rW:=\Ind_{T\subset B}^G(\mF)^{\rW,\on{triv}}
=(p'_G)_!\big((p'_T)^*\mF'\otimes \IC(\rW\backslash S,\barQ)\big)[-\dim T],
\eeq
(here $p_G',p_T'$ are the morphisms in~\eqref{quotient diagram})
for the summand in~\eqref{summand} corresponding to the 
trivial representation $(\on{triv},V_{\on{triv}}=\barQ)$.

\end{definition}

In the case when $\mF$ is a $\rW$-equivariant perverse local system on $T$,
we have the following description of~\eqref{W-inv summand}:
Let $q^{rs}:T^{rs}\to \rW\backslash\backslash T^{rs}$
and $c^{rs}:G\to \rW\backslash\backslash T^{rs}$ be the restriction of 
the maps in~\eqref{G-S resolution} to the regular semi-simple locus.
As $q^{rs}$ is an \'etale covering, 
the restriction of 
$\mF$ to $T^{rs}$ descends to a perverse local system $\mF'$ on 
$\rW\backslash\backslash T^{rs}$ and we have 
\[\Ind_{T\subset B}^G(\mF)^\rW\is j_{!*}(c^{rs})^*\mF'[\dim G-\dim T].\]

Let 
$P$ be a standard
 parabolic subgroup containing 
$B$ and let 
 $L$ be the unique Levi subgroup of $P$ containing $T$ with Borel subgroup 
 $B_L=B\cap L$.
Let $\rW_L$ be the Weyl group of $L$, which is naturally a subgroup of 
$\rW$.

\begin{proposition}\label{Mackey formula}
Let $\mF\in\sD_\rW(T)^\heartsuit$.
(1)
We have a canonical isomorphism 
\[\Res_{L\subset P}^G\circ\Ind_{T\subset B}^G(\mF_T)\is\Ind^\rW_{\rW_L}\Ind_{T\subset B_L}^L(\mF_T)\]
which is compatible with the nature $\rW$-actions on both sides.
(2) There is a canonical isomorphism 
\[\Res_{L\subset P}^G\circ\Ind_{T\subset B}^G(\mF)^\rW\is\Ind_{T\subset B_L}^L(\mF)^{\rW_L}.\]
\end{proposition}
\begin{proof}
Part (2) follows from part (1) by taking $\rW$-invariants on both sides.
Part (1) is proved in \cite[Theorem 2.7]{BK2}.
\footnote{In \emph{loc. cit.} the authors assumed $\mF$ is irreducible with support a $\rW$-stable sub tours in $T$.
This is because the proof makes use of the isomorphism $\Ind_{T\subset B}^G(\mF)\is\Ind_{T\subset B}^G(w^*\mF)$ which was constructed only for those $\mF$ 
satisfying the assumption above
(see \cite[Theorem 2.5(3)]{BK1}). Now, with Proposition \ref{W action}, the same argument 
works for arbitrary $\rW$-equivariant perverse sheaves on $T$.
}

\end{proof}
\section{Characterization of central complexes}\label{central complexes}
\subsection{The scheme of tame characters}\label{C(T)}
Let $\pi_1(T)$ be the \'etale fundamental group of $T$ 
and let $\pi_1(T)^t$ be its tame quotient.
A continuous character 
$\chi:\pi_1(T)\to\barQ^\times$ is called \emph{tame} if 
it factors though $\pi_1(T)^t$.
For any continuous character $\chi:\pi_1(T)\to\barQ^\times$, 
we denote by $\mL_\chi$  the corresponding rank one 
local system on $T$. 
A rank one local system $\mL$ on $T$ is called \emph{tame}
if $\mL\is\mL_\chi$ for a tame character $\chi$.

In \cite{GL}, a $\barQ$-scheme $\calC(T)$ is defined, whose 
$\barQ$-points are in bijection with 
tame characters of $\pi_1(T)$. There is decomposition 
\beq\label{component}
\calC(T)=\bigsqcup_{\chi_f\in\calC(T)_f}\{\chi_f\}\times\calC(T)_\ell
\eeq into connected components, where 
$\calC(T)_f\subset\calC(T)$ is the subset consisting of tame characters of 
order prime to $\ell$ and 
$\calC(T)_\ell$ is the connected component of $\calC(T)$ containing 
the trivial character.
It is shown in \emph{loc. cit.} that $\calC(T)$ is Noetherian and regular and there is an isomorphism
\[\calC(T)_\ell\is\Spec(\barQ\otimes_{\bZ_\ell}\bZ_\ell[[x_1,...,x_r]]).\] 
In addition, the $\barQ$-points of $\calC(T)_\ell$ are in bijection 
with pro-$\ell$ characters of $\pi_1(T)$ (i.e. characters of 
the pro-$\ell$ quotient $\pi_1(T)_\ell$ of $\pi_1(T)^t$).

\subsection{The group $\rW_\chi$}
The Weyl group $\rW$ acts naturally on $\calC(T)$.  
For any $\chi\in\calC(T)(\barQ)$
we denote by $\rW'_\chi$ the stabilizer of $\chi$ in $\rW$.
Equivalently, $\rW_\chi$ is the group consisting of 
$w\in\rW$ such that $w^*\mL_\chi\is\mL_\chi$.
Let
$\rW_\chi$ to be the subgroup of $\rW'_\chi$ generated by reflections 
$s_\alpha$ satisfying the following property:
Let $\check\alpha:\bG_m\to T$ be the coroot corresponding to $\alpha$.
The pullback $(\check\alpha)^*\mL_\chi$ is isomorphic to the trivial local system on 
$\bG_m$.

\begin{example}\label{W_chi}
Let $G=\SL_2$ and let 
$\chi:\pi_1(T)^t\to\{\pm 1\}$ be the character corresponding 
to the tame covering $T\to T, x\to x^2$.
 Then we have $\rW'_\chi=\rW$ but $\rW_\chi=e$ is trivial.
\end{example}

\subsection{Central complexes}\label{central objects}
Let $\mF\in\sD_\rW(T)$.
For any tame character $\chi\in\calC(T)(\barQ)$ the stabilizer 
$\rW'_\chi$, hence also its subgroup $\rW_\chi$, acts naturally on the 
\'etale cohomology groups
$\oH_c^*(T,\mF_T\otimes\mL_\chi)$ (resp. $\oH^*(T,\mF_T\otimes\mL_\chi)$).

\begin{definition}\label{tw descent}
A $\rW$-equivariant complex $\mF\in\sD_\rW(T)$ is called 
\emph{central} (resp. $*$-\emph{central}) if the following holds: for any tame character $\chi\in\calC(T)(\barQ)$, the group $\rW_\chi$ acts 
on \[\ \ \ \ \ \oH_c^*(T,\mF\otimes\mL_\chi)\ \ \ \ \ \  (\text{resp.}\ \  \oH^*(T,\mF\otimes\mL_\chi))\] 
through the sign character $\sign_{\rW}:\rW_\chi\to\{\pm1\}$.
\end{definition}

\quash{
\begin{example}\label{KL}
Let $i:\{e\}\to T$ be the embedding of the unit element. Then 
$i_*\barQ\otimes\sign$ is central.
\end{example}

\begin{remark}
Note that the trivial local system $\mF=\overline{\bbQ}_{\ell, T}$ on $T$
with the canonical $\rW$-equivariant structure 
is neither central nor $*$-central.
\end{remark}
}
\quash{
\begin{remark}\label{equivalent def}
The definition above is equivalent to the following:
$\mF$ is central ($*$-central) if, for any tame character 
$\chi$, 
$\rW_\chi$-action on $\oH_c^*(T,\mF\otimes\sign\otimes\mL_\chi)$
(resp. $\oH^*(T,\mF\otimes\sign\otimes\mL_\chi)$)
is trivial.
\end{remark}}

\subsection{Mellin transforms}\label{Mellin}
In \cite{GL}, the authors constructed the \emph{Mellin transforms}
\[\calM_!:\sD(T)\to D^b_{coh}(\calC(T))\]
\[\calM_*:\sD(T)\to D^b_{coh}(\calC(T))\]
with the following properties:
\begin{enumerate}
\item
Let $\chi\in\calC(T)(\barQ)$ and $i_\chi:\{\chi\}\to\calC(T)$ be the natural inclusion. We have 
\beq\label{fiber}
R\Gamma_c(T,\mF\otimes\mL_\chi)\is i^*_\chi\calM_!(\mF)
\eeq
\[R\Gamma(T,\mF\otimes\mL_\chi)\is i^*_\chi\calM_*(\mF).\]
\item
We have natural isomorphism $\bbD_{}(\calM_!(\mF))\is\on{inv}^*\calM_*(\bbD(\mF))$. 
\item 
The functor $\calM_*$ is t-exact with respect to the perverse $t$-structure on 
$\sD(T)$ and the natural $t$-structure on $D^b_{coh}(\calC(T))$.
Moreover, for any $\mF\in \sD(T)$, $\mF$ is perverse if and only if 
$\calM_*(\mF)$ is a coherent complex in degree zero.

\item 
We have \[\calM_!(\mF*_!\mF')\is\calM_!(\mF)\otimes\calM_!(\mF')\]
\[\calM_*(\mF* \mF')\is\calM_*(\mF)\otimes\calM_*(\mF')\]
where $\mF*_!\mF'=m_!(\mF\boxtimes\mF')$, $\mF*\mF'=m_*(\mF\boxtimes\mF')$, and $m:T\times T\to T$
is the multiplication map.
\item For any $\chi\in\calC(T)_f$ we have 
\[\calM_*(\mF)|_{\{\chi\}\times\calC(T)_\ell}\is
\calM_*(\mF\otimes\mL_\chi)|_{\calC(T)_\ell}.\]
\item The Mellin transforms restricts to an equivalence 
\beq\label{inverse Mellin}
\calM_!,\calM_*:\sD(T)_{\on{mon}}\is D^b_{coh}(\calC(T))_f
\eeq
between the full subcategory $\sD(T)_{\on{mon}}$ of monodromic $\ell$-adic complexes on $T$ and 
the full subcategory $D^b_{coh}(\calC(T))_f$ of coherent complexes on $\calC(T)$ with 
finite support.
\end{enumerate}

The Weyl group $\rW$ acts naturally on $\calC(T)$ and it follows from the 
construction of Mellin transforms that for $\mF\in\sD_\rW(T)$, 
the $\rW$-equivariant structure on $\mF$ gives rise to a 
$\rW$-equivariant structure on
$\calM_!(\mF)$ (resp. $\calM_*(\mF)$), such that for any $\chi\in\calC(T)(\barQ)$,
the isomorphism 
\eqref{fiber} above 
is compatible with the natural 
$\rW'_\chi$-actions. 

We have the following characterization of central complexes:

\begin{proposition}\label{characterization}
Let $\mF\in\sD_\rW(T)$ be a $\rW$-equivariant complex.
The following are equivalent:
\begin{enumerate}
\item $\mF$
 is central.
\item 
For any $\chi\in\calC(T)(\barQ)$ the action of $\rW_\chi$ on 
$i^*_\chi\calM_!(\mF\otimes\sign)
$
is trivial. 
\\ 
Assume further that $\rW_\chi=\rW_\chi'$ for all $\chi\in\calC(T)(\barQ)$, then above statements are equivalent to
\item
The restriction of the Mellin transfrom 
$
\calM_!(\mF\otimes\sign)\in D_{coh}^b(\calC(T))
$ to each 
connected component 
$\calC(T)_{\ell,\chi_f}:=\{\chi_f\}\times\calC(T)_\ell$ of 
$\calC(T)$ (see~\eqref{component})
descends to the quotient $\rW_{\chi_f}\backslash\backslash\calC(T)_{\ell,\chi_f}$.\footnote{Note that $\calC(T)_{\ell,\chi_f}$ is stable under the $\rW_{\chi_f}$-action.}
\end{enumerate}
The same is true for $*$-central complexes if we replace 
$\calM_!$ by $\calM_*$.
\end{proposition}
\begin{proof}
$(1)\Leftrightarrow(2)$
follows from the property (1) of Mellin transform above.
Assume
$\rW_\chi=\rW_\chi'$ for all $\chi$.
Then for any $\chi\in\calC(T)_{\ell,\chi_f}$ we have 
$\rW_\chi=\rW_{\chi}'\subset\rW_{\chi_f}'=\rW_{\chi_f}$
and it follows that 
the stabilizer of $\chi$ in $\rW_{\chi_f}$ is equal to $\rW_\chi$
and 
$(2)\Leftrightarrow(3)$
follows from the descent criterion for coherent complexes in
\cite[Theorem 1.3]{N} (or \cite[Lemma 5.3]{C}).
\end{proof}

\quash{
\begin{corollary}
Let $\mF,\mF'$ be two central complexes.
Then $(\mF*_!\mF')\otimes\sign$ is again a central complexes.
The same is true for $*$-central complexes if we replace 
$*_!$ by $*$.
\end{corollary}
\begin{proof}
By the proposition above, is suffices to show that 
the $\rW_\chi$-action on
$i^*_\chi\calM_{!}((\mF*_!\mF'))$ is trivial.
This follows from the isomorphism 
\[i^*_\chi\calM_{!}((\mF*_!\mF'))\is
i^*_\chi(\calM_{!}(\mF)\otimes\calM_{!}(\mF'))\is
i^*_\chi\calM_{!}(\mF)\otimes i^*_\chi\calM_{!}(\mF')\]
and the fact that the $\rW_\chi$-actions on 
$i^*_\chi\calM_{!}(\mF)$ and $i^*_\chi\calM_{!}(\mF')$ are given by the 
sign character.
\end{proof}

The corollary above shows that 
the map
\beq\label{monodical structure}
(\mF',\mF)\to(\mF'*_!\mF)\otimes\sign\ \ \ \ (\text{resp.}\ \ \ (\mF',\mF)\to(\mF'*\mF)\otimes\sign)\eeq
defines a symmetric monodical structure 
on the category of central complexes (resp. $*$-central complexes). 

}

\subsection{Examples of tame central local systems}\label{tame central}
Consider the quotient map
$\pi_\chi:\calC(T)_\ell\to\rW_\chi\backslash\backslash\calC(T)_\ell$.
Let $0\in\calC(T)(\barQ)$ be the trivial character and let 
$\pi_\chi(0)$ be its image in $\rW_\chi\backslash\backslash\calC(T)_\ell$.
 Let $\mO_{\pi_\chi(0)}$ be the structure sheaf of the point $\pi_\chi(0)$
 and we define 
\beq\label{descent of R_chi}
\calR_\chi:=\pi_\chi^*\mO_{\pi_\chi(0)}
\eeq
which a $\rW_\chi$-equivariant coherent sheaf on $\calC(T)_\ell$.
Note that, as $\rW_\chi$ is normal subgroup of $\rW'_\chi$, 
the $\rW_\chi'$-action on $\calC(T)_\ell$ descends to a $\rW_\chi'$-action 
on the quotient $\rW_\chi\backslash\backslash\calC(T)_\ell$ fixing $\pi_\chi(0)$ and it follows that 
$\calR_\chi$ has a canonical $\rW_\chi'$-equivariant structure.

We will regard $\calR_\chi$ as a coherent sheaf on $\calC(T)$ supported at the 
component $\calC(T)_\ell=\{0\}\times\calC(T)_\ell$ and 
define 
$\calS_\chi:=m_\chi^*(\calR_\chi)$ where
$m_\chi:\calC(T)\to\calC(T)$ be the morphism of translation by $\chi$.
Since $m_\chi$ intertwines the $\rW_\chi'$-action on $\calC(T)$, 
$\calS_\chi$ is $\rW_\chi'$-equivariant,
moreover, there is natural isomorphism 
\beq\label{w action}
w^*\calS_\chi\is\calS_{w^{-1}\cdot\chi}
\eeq for any $w\in\rW$.\footnote{Indeed, we have 
$w^*\calS_\chi\is w^*m_\chi^*\calR_\chi\is m_{w^{-1}\chi}^*w^*\calR_\chi\is m_{w^{-1}\chi}^*\calR_{w^{-1}
\chi}=\calS_{w^{-1}\chi}$, where we use the observation that 
$w^{-1}\rW_\chi w=\rW_{w^{-1}\chi}$ and hence $w^*\calR_\chi\is\calR_{w^{-1}\chi}$.}
For any $\rW$-orbit $\theta$ in $\calC(T)$, we define 
the following $\rW$-equivariant coherent sheaf 
with finite support
\[\mS_\theta=\bigoplus_{\chi\in\theta}\calS_\chi\]
where the $\rW$-equivariant structure is given by the 
isomorphisms in~\eqref{w action}. 

Finally, we define the following $\rW$-equivariant perverse local systems on $T$:
\[\mE^!_\theta:=\calM_!^{-1}(\calS_\theta)\otimes\sign\]
\[\mE_\theta:=\calM_*^{-1}(\calS_\theta)\otimes\sign.
\]
Here $\calM_!^{-1},\calM_*^{-1}$ are the inverse of the Mellin transforms in~\eqref{inverse Mellin}.\footnote{Note that $\calS_\theta$ has finite support and hence the inverse of the Mellin transform is well-defined by~\eqref{inverse Mellin}}
\begin{lemma}
$\mE^!_\theta$ is central and 
$\mE_\theta$ is $*$-central.
\end{lemma}
\begin{proof}
Since $\calS_\theta$ is supported on 
$\theta^{-1}=\{\chi^{-1}|\chi\in\theta\}$,
by Proposition \ref{characterization}, it suffices to show that the action of $\rW_{\chi^{-1}}$
(note that $\rW_{\chi^{-1}}=\rW_\chi$) on
the fiber 
\[i_{\chi^{-1}}^*\calM_!(\mE^!_\theta\otimes\sign)\is i_{\chi^{-1}}^*\calM_!(\calM^{-1}_!(\calS_\theta))
\is i_{\chi^{-1}}^*\calS_\theta\is i_{\chi^{-1}}^*\calS_\chi\]
is trivial. 
Let $\calC(T)_{\ell,\chi^{-1}}\subset\calC(T)$ be the component containing 
$\chi^{-1}$. Then the translation map 
$m_\chi$ restricts to a map
$m_\chi:\calC(T)_{\ell,\chi^{-1}}\to\calC(T)_\ell$, moreover, 
we have the following commutative diagram
\[\xymatrix{\calC(T)_{\ell,\chi^{-1}}\ar[r]^{m_\chi}\ar[d]^{\pi_{\ell,\chi^{-1}}}&\calC(T)_{\ell}\ar[d]^{\pi_\chi}\\
\rW_{\chi^{-1}}\backslash\backslash\calC(T)_{\ell,\chi^{-1}}\ar[r]^{\bar m_\chi}&\rW_\chi\backslash\backslash\calC(T)_{\ell}}\]
where $\bar m_\chi$ is the descent of $m_\chi$.
It follows that the restriction of 
$\calS_\chi=m_\chi^{*}\pi_\chi^{*}\mO_{\pi_\chi(0)}\is\pi_{\ell,\chi^{-1}}^*\bar m_\chi^*\mO_{\pi_\chi(0)}$ to $\calC(T)_{\ell,\chi^{-1}}$
descends to $\rW_{\chi^{-1}}\backslash\backslash\calC(T)_{\ell,\chi^{-1}}$
and it implies the action of 
$\rW_{\chi^{-1}}$ on $i_{\chi^{-1}}^*\calS_\chi$ is trivial.
\end{proof}

\section{$\rho$-Bessel sheaves}\label{gamma sheaves}

\subsection{Hypergeometric sheaves}\label{Hypergeometric sheaves}
Let $\underline\lambda=\{\lambda_1,...,\lambda_r\}\subset\bX_\bullet(T)$ be a collection of possible repeated cocharacters.
Consider the following maps
\[\bG_a\stackrel{\tr}\longleftarrow\bG_m^r\stackrel{\pr_{\underline\lambda}}\longrightarrow T\]
where $\tr(x_1,...,x_r)=\sum_{i=1}^r x_i$ and $\pr_{\underline\lambda}(x_1,...,x_r)=\prod_{i=1}^{r}\lambda_i(x_i)$.
Consider the following complexes:
\beq\label{Hyper}
\Phi_{\underline\lambda}:=\pr_{\underline\lambda,!}\tr^*\mL_\psi[r],\ \ \ \ 
\Phi_{\underline\lambda}^*:=\pr_{\underline\lambda,*}\tr^*\mL_\psi[r]
\eeq
in $D(T)$.
Note that we have a natural forget supports map 
\beq\label{! to *}
\Phi_{\underline\lambda}\to\Phi_{\underline\lambda}^*.
\eeq
Following Katz \cite{K}, we will call the complexes in \eqref{Hyper}
hypergeometric sheaves.

Let $\sigma:T\to\bG_m$ be a character $\lambda$. 
A cocharacter is called $\sigma$-positive if $\langle\sigma,\lambda\rangle$ is positive.

\begin{prop}\label{cleanness of gamma}
Let $\underline\lambda=\{\lambda_1,...,\lambda_r\}\subset\bX_\bullet(T)$
be a collection of cocharacters.
\begin{enumerate}
\item
Assume that each $\lambda_i$ is nontrivial.
Then $\Phi_{\underline\lambda}$ 
and $\Phi_{\underline\lambda}^*$ are perverse sheaves.
\item
Assume that each $\lambda_i$ are $\sigma$-positive.
Then the map $\Phi_{\underline\lambda}\to\Phi_{\underline\lambda}^*$ in~\eqref{! to *} is an isomorphism 
and $\Phi_{\underline\lambda}\is\Phi_{\underline\lambda}^*$
is a perverse local system over the image of $p_{\underline\lambda}$, which is a subtorus of $T$.
\end{enumerate}
\end{prop}
\begin{proof}
This is \cite[Theorem 4.2]{BK2}, and \cite[Appendix B]{CN}
\end{proof}

\begin{example}
For $T=\bG_m$ and $\lambda_i=\id:\bG_m\to\bG_m$
for all $i$. Then each $\lambda_i$ is $\sigma$-positive for $\sigma=\id$, and the 
corresponding  perverse local system $\Phi_{\underline\lambda}$
on 
$\bG_m$ is the Kloosterman sheaf considered by Deligne in \cite{D}.
\end{example}

Let $\rS_{\underline\lambda}$ be the subgroup 
of the symmetric group $\rS_r$ consisting of 
permutations $\tau$ such that for all $i\in\{1,...,r\}$,
we have $\lambda_{\tau(i)}=\lambda_i$.
We have the following result due to Deligne \cite[Proposition 7.20]{D}:
\begin{prop}\label{sign}
The group $\rS_{\underline\lambda}$
acts on $\Phi_{\underline\lambda}$ (resp. $\Phi_{\underline\lambda}^*$) via the sign character 
$\sign_r:\rS_r\to\{\pm1\}$.
\end{prop}

\subsection{$\rho$-Bessel sheaves on $T$}
We recall the construction of Braverman-Kazhdan's $\rho$-Bessel sheaves on $T$ 
attached to a representation of the dual group $\check G$.
Let $\rho:\check G\to \GL(V_\rho)$ be a $r$-dimensional complex representation of $\check G$.
The restriction of $\rho$ to $\check T$ is diagonalizable and there exists 
a collection of weights 
\[\underline\lambda=\{\lambda_1,...,\lambda_r\}\subset\bX^\bullet(\check T):=\Hom(\check T,\bC^\times)\]
such that there is an eigenspace decomposition 
\[V_\rho=\bigoplus_{i=1}^r V_{\lambda_i}\]
of $V_\rho$, where $\check T$ acts on $V_{\lambda_i}$ via the character $\lambda_i$.

We can regard $\underline\lambda$ as collection of 
cocharacters of $T$ using the the canonical isomorphism
$\bX^\bullet(\check T)\is\bX_\bullet(T)$
and we denote by 
\[\Phi_{T,\rho}:=\Phi_{\underline\lambda},\ \ \ \Phi^*_{T,\rho}:=\Phi_{\underline\lambda}^*\]
the hypergeometric sheaves associated to 
$\underline\lambda$ in~\eqref{Hyper}. We will call 
them 
$\rho$-Bessel sheaves.

Following \cite{BK2} (see also \cite{CN}), we shall construct a $\rW$-equivariant structure on $\rho$-Bessel sheaves.
Let $\{\lambda_{i_1},...,\lambda_{i_k}\}$ be the distinct cocharacters appearing in $\ul$ and 
$m_l$ be the multiplicity of $\lambda_{i_l}\in\ul$.
Let $A_m=\{\lambda_j|\lambda_j=\lambda_{i_m}\}$.
Then we have $\{\lambda_1,...,\lambda_r\}=A_1\sqcup...\sqcup A_k$.
The symmetric group on $r$-letters $\mathrm S_r$ acts naturally on $\{\lambda_1,...,\lambda_r\}$
and we define  $\rS_\ul=\{\sigma\in\rS_r|\sigma(A_i)=A_i\}$. 
There is
a canonical isomorphism 
\[\rS_\ul\is\rS_{m_1}\times\cdot\cdot\cdot\times\rS_{m_k}.\]
Define $\rS'_\ul=\{\eta\in\rS_r|\text{such that}\ \eta(A_i)=A_{\tau(i)}\ \text{for a}\ \tau\in\rS_k\}$.
We have a natural map $\pi_k:\rS_\ul'\ra\rS_k$ sending $\eta$ to $\tau$. The 
kernel of $\pi_k$ is isomorphic to $\rS_\ul$ and 
its image, denote by $\rS_{k,\ul}$, 
consists of $\tau\in\rS_k$ such that $m_i=m_{\tau(i)}$. 
In other words, there is a short exact sequence 
\[0\ra\rS_\ul\ra\rS_\ul'\stackrel{\pi_k}\ra\rS_{k,\ul}\ra 0.\] 

Notice that the Weyl group $\rW$ acts on $\{\lambda_{i_1},...,\lambda_{i_k}\}$
and the induced map $\rW\ra\rS_k$ has image $\rS_{k,\ul}$.
So we have a map $\rho:\rW\ra\rS_{k,\ul}$. Pulling back the short exact sequence 
above along $\rho$, we get an extension of $W'$ of $W$ by $\rS_\ul$
\[0\ra\rS_\ul\ra W'\ra W\ra 0\] 
where an element in $w'\in W'$ consists of pair $(w,\eta)\in\rW\times\rS_\ul'$ such
that $\rho(w)=\pi_k(\eta)\in\rS_{k,\ul}$.

The group $\rW'$ acts on $\bG_m^r$ (resp. $T$) via the composition of 
the action of $\rS_r$ (reps. $\rW$) with the natural projection $\rW'\ra\rS_\ul'\subset\rS_r$ 
(resp. $\rW'\ra\rW$) and the map $\pr_\ul:\bG_m^r\ra T$ and $\tr:\bG_m^r\ra\bG_a$ is $W'$-equivaraint
where $\rW'$ acts trivially on $\bG_a$. 

Since $\Phi_{T,\rho}=\pr_{\ul,!}\tr^*\mL_\psi[r]$
(resp. $\Phi_{T,\rho}^*=\pr_{\ul,*}\tr^*\mL_\psi[r]$), the discussion above 
implies for each $w'=(w,\eta)\in\rW'$ there is an isomorphism 
\beq\label{i'_w'}
i'_{w'}:\Phi_{T,\rho}\is w^*\Phi_{T,\rho}
\eeq 
\[(\text{resp.}\  i'_{w'}:\Phi^*_{T,\rho}\is w^*\Phi^*_{T,\rho}).\]
We define 
\beq\label{i_w'}
i_{w'}=\sign_\rW(w)\on{sign}_r(\eta)i'_{w'}:\Phi_{T,\rho}\is w^*\Phi_{T,\rho}
\eeq
\[(\text{resp.}\ \ \  i_{w'}=\sign_\rW(w)\on{sign}_r(\eta)i'_{w'}:\Phi^*_{T,\rho}\is w^*\Phi^*_{T,\rho})
\]
where $\on{sign}_r$ is the sign character of $\rS_r$.
It follows from Proposition \ref{sign} that the isomorphism $i_{w'}$ depends only on $w$. Denote the resting isomorphism by $i_w$,
then the data $(\Phi_{T,\rho},\{i_w\}_{w\in\rW})$ (resp. $(\Phi^*_{T,\rho},\{i_w\}_{w\in\rW})$) defines an object in $\sD_\rW(T)$ which, by abuse of notation, we still denote by
$\Phi_{T,\rho}$ (resp. 
$\Phi_{T,\rho}^*$).


\begin{example}
Consider the case $G=\GL_r$ and 
$\rho$ is the standard representation of 
$\check G=\GL_r(\bbC)$.
We have $\Phi_{T,\rho}=\tr^*\mL_\psi[r]\in\sD_\rW(T)$,
where $\tr:\bG_m^r\to \bG_a$ is the trace map. 
\end{example}
\quash{
Here the first main result of the article:
\begin{thm}\label{gamma is central}
For any tame character 
$\chi$ of $T$, the action of $\rW'_\chi$ on 
\[\oH_c^*(T,\Phi_{T,\rho}\otimes\mL_\chi)  \ \ \ \ (\text{resp.}\ \  \oH^*(T,\Phi^*_{T,\rho}\otimes\mL_\chi))\] is trivial. 
In particular, 
the gamma sheaf $\Phi_{T,\rho}$ (resp. $\Phi^*_{T,\rho}$)
 is central (resp. $*$-central)
\end{thm}

We can restate Theorem \ref{gamma is central} in the following form:
Note that the assignment $\rho\in\on{Rep}(\check G)$ to $\Phi_{T,\rho}^!$ (resp. $\Phi^*_{T,\rho}$) defines a functor
\beq
\mathrm{BK_!}:\on{Rep}(\check G)\to D(\rW\backslash T)
\eeq
\[
(\text{resp.}\ \ \mathrm{BK_*}:\on{Rep}(\check G)\to D(\rW\backslash T))
.\]

\begin{thm}
\end{thm}
}

\section{Proof of Theorem \ref{main 1}}
Recall the maps
\[\bG_a\stackrel{\tr}\longleftarrow\bG_m^r\stackrel{\pr_{\underline\lambda}}\longrightarrow T.\]
Let $\chi\in\calC(\bG_m^r)(\barQ)$ be a tame character.
The permutation action of $\rS_r$ on $\bG_m^r$ induces an action of
$\rS_r$ on $\calC(\bG_m^r)$ and let
$\rS_{r,\chi}$ be the stabilizer of $\chi$ in $\rS_r$.
The pullback $\tr^*\mL_\psi$ is naturally $\rS_r$-equivariant and 
we have a natural $\rS_{r,\chi}$-action on 
the cohomology 
$\oH^*(\bG_m^r,\tr^*\mL_\psi\otimes\mL_\chi)$
(resp. $\oH_c^*(\bG_m^r,\tr^*\mL_\psi\otimes\mL_\chi)$).

\begin{lemma}\label{S_r-action}
The $\rS_{r,\chi}$-action on $\oH_c^*(\bG_m^r,\tr^*\mL_\psi\otimes\mL_\chi)$ is given by $\sign_r:\rS_{r,\chi}\to\{\pm1\}$.
\end{lemma}
\begin{proof}
Let 
$\sigma=(i,j)\in\rS_{r,\chi}$ 
be a simple reflection. It suffices to show that 
$\sigma$
acts on $\oH^*(\bG_m^r,\tr^*\mL_\psi\otimes\mL_\chi)$ by $-1$, or equivalently, the $\sigma$-invariant 
$\oH_c^*(\bG_m^r,\tr^*\mL_\psi\otimes\mL_\chi)^\sigma$
is zero.
Write $\mL_\chi=\mL_1\boxtimes\cdot\cdot\cdot\boxtimes\mL_r$,
where each $\mL_k$ is a tame local system on $\bG_m$.
Note that $\sigma^*\mL_\chi\is\mL_\chi$ implies 
$\mL_i\is\mL_j:=\mL$. Consider the quotient map
\[q:\bG_m^r\to\sigma\backslash\backslash\bG_m^r\is\bbA^1\times\bG_m\times\prod_{k\in\{1,...,r\}, k\neq i,j}\bG_m,\]
given by
\[q(x_1,...,x_r)=(x_i+x_j,x_ix_j,\prod_{k\in\{1,...,r\}, k\neq i,j} x_k).\]
Using the fact that $\mL\boxtimes\mL\is m^*\mL$, where $m:\bG_m^2\to\bG_m$ is the multiplication map, we see that 

\[\tr^*\mL_\psi\otimes\mL_\chi\is q^*\big(\mL_\psi\boxtimes\mL\boxtimes(\tr^*\mL_\psi\otimes\prod_{k\neq i,j}\mL_k)\big).\]
The permutation $\sigma$ acts naturally on 
$q_*(\tr^*\mL_\psi\otimes\mL_\chi)$ and it follows from the isomorphism above that 
\[(q_*(\tr^*\mL_\psi\otimes\mL_\chi))^\sigma\is
(q_*q^*\big(\mL_\psi\boxtimes\mL\boxtimes(\tr^*\mL_\psi\otimes\prod_{k\neq i,j}\mL_k)\big))^\sigma\is
\mL_\psi\boxtimes\mL\boxtimes(\tr^*\mL_\psi\otimes\prod_{k\neq i,j}\mL_k).\]
This implies 
\[\oH_c^*(\bG_m^r,\tr^*\mL_\psi\otimes\mL_\chi)^\sigma\is\oH_c^*(\sigma\backslash\backslash\bG_m^r,
\mL_\psi\boxtimes\mL\boxtimes(\tr^*\mL_\psi\otimes\prod_{k\neq i,j}\mL_k))=0\]
where the last equality follows from the cohomology vanishing 
$\oH_c^*(\bbA^1,\mL_\psi)=0$.
The lemma follows.

\end{proof}

To proceed,
let $\chi\in\calC(T)(\barQ)$ and 
let $\chi'=\pr_{\underline\lambda}^*\chi\in\calC(\bG_m^r)(\barQ)$ be the 
pull back of $\chi$.
We have $\pr_{\underline\lambda}^*\mL_{\chi}\is\mL_{\chi'}$ and
\beq\label{action}
\oH_c^*(T,\Phi_{T,\rho}\otimes\mL_\chi)\is
\oH_c^*(T,\pr_{\underline\lambda,!}\tr^*\mL_\psi[r]\otimes\mL_\chi)\is
\oH_c^*(\bG_m^r,\tr^*\mL_\psi[r]\otimes\mL_{\chi'})
\eeq
where the second isomorphism comes from the projection formula.
Let $w\in\rW'_\chi$ and choose a lift $w'=(w,\eta)\in\rW'$ of $w$.
Note that the following diagram is commutative
\[\xymatrix{\bG_m^r\ar[r]^{\pr_{\underline\lambda}}\ar[d]^{\eta}&T\ar[d]^w\\
\bG_m^r\ar[r]^{\pr_{\underline\lambda}}&T}.\]
It follows that 
$\eta^*\mL_{\chi'}\is\eta^*\pr_{\underline\lambda}^*\mL_\chi\is
\pr_{\underline\lambda}^*w^*\mL_\chi\is\pr_{\underline\lambda}^*\mL_\chi\is\mL_{\chi'}$, that is, $\eta\in\rS_{r,\chi'}$.
Moreover, we have the following commutative 
diagram
\[\xymatrix{\oH_c^*(T,\Phi_{T,\rho}\otimes\mL_\chi)\ar[r]^{}\ar[d]^{i_{w'}'}&\oH_c^*(\bG_m^r,\tr^*\mL_\psi[r]\otimes\mL_{\chi'})\ar[d]^\eta\\
\oH_c^*(T,\Phi_{T,\rho}\otimes\mL_\chi)
\ar[r]&\oH_c^*(\bG_m^r,\tr^*\mL_\psi[r]\otimes\mL_{\chi'})
}\]
where the horizontal arrows are 
the isomorphism~\eqref{action}, the left vertical arrow is the isomorphism induced by 
the isomorphism the $i'_{w'}$
in
~\eqref{i'_w'}, and the right vertical arrow is the 
action of $\eta$ on $\oH_c^*(\bG_m^r,\tr^*\mL_\psi[r]\otimes\mL_{\chi'})$ coming from the $\rS_{r,\chi'}$-equivariant structure 
on $\tr^*\mL_\psi[r]\otimes\mL_{\chi'}$.
Therefore, by Lemma \ref{S_r-action}, we see that $i'_{w'}=\sign_r(\eta)$
and it follow from the definition of $\rW$-equivariant structure 
of $\rho$-Bessel sheaf in~\eqref{i_w'} that the action of $w$ on 
the cohomology group $\oH_c^*(T,\Phi_{T,\rho}\otimes\mL_\chi)$
is given by 
\[i_{w'}=\sign_\rW(w)\sign_r(\eta)i'_{w'}=\sign_\rW(w)\sign_r(\eta)\sign_r(\eta)=\sign_\rW(w).\]
The proof of Theorem \ref{main 1} is complete.

\section{Proof of Theorem \ref{main 2}}\label{GL_n case}
In this section we shall prove  
the vanishing conjecture for $G=\GL_n$.

\subsection{The $\GL_2$-example}
Let us first consider the simple but important case $G=\GL_2$. 
Let $\mF\in\sD_\rW(T)$ be a $\rW$-equivariant complex on $T=\bG_m^2$
and let $\Phi_\mF=\Ind_{T\subset B}^G(\mF)^\rW$.
Note that, for $x\in G\setminus B$, the coset $Ux\subset G^{\text{reg}}$ consists of regular elements. Note also that 
the Grothendieck-Springer simultaneous resolution 
is Cartesian over $G^{\on{reg}}$:
\beq\label{G-S resolution}
\xymatrix{\widetilde G^{\on{reg}}\is G^{\on{reg}}\times_{\rW\backslash\backslash T}T\ar[r]^{\ \ \ \ \ \ \ \ \ \tilde q}\ar[d]^{\tilde c\ \ }&T\ar[d]^q\\
G^{\on{reg}}\ar[r]^c&\rW\backslash\backslash T}
\eeq
All together, we obtain 
\[\oH_c^*(Ux,\Phi_\mF)\is\oH_c^*(Ux,\tilde c_!\tilde q^*\mF)^\rW\is
\oH_c^*(Ux,c^*q_!\mF)^\rW\]
Under the identification $\rW\backslash\backslash T\is\bG_a\times\bG_m$,
the maps $c$ (resp. $q$) is given by $c(g)=(\on{tr}_G(g),\det(g))$
(resp. $q(t)=(\tr_T(t),\det(t))$), and 
a direct calculation shows that 
$c$ restricts to an isomorphism 
\beq\label{key iso GL_2}
c:Ux\is \bG_a\times\on{det}(x)\subset\bG_a\times\bG_m\is \rW\backslash\backslash T.
\eeq
Thus we have
\beq\label{key iso}
\oH_c^*(Ux,\Phi_\mF)\is\oH_c^*(Ux,c^*q_!\mF)^\rW\is (m_!\mF)^\rW|_{\det(x)},
\eeq
where $m=\pr_{\bG_m}\circ q:T=\bG_m^2\to\bG_m, (x,y)\to xy$.
As $\det|_{G\setminus B}:G\setminus B\to\bG_m$ is surjective, it follows from~\eqref{key iso}
that 
$\oH_c^*(Ux,\Phi_\mF)=0$ for all $x\in G\setminus B$
 if and only if $\rW$ acts on $m_!\mF$ via the sign character.
 We claim that the later property of $m_!\mF$ is equivalent 
 to $\mF$ being central.  Thus we conclude that 
 $\oH_c^*(Ux,\Phi_\mF)=0$ for all $x\in G\setminus B$
 if and only if $\mF$ is central. This completes the proof of the vanishing conjecture for 
$\GL_2$.

To prove the claim we observe that
$\rW$ acts on $m_!\mF$ via the sign character if and only if 
$(m_!\mF)^\rW=0$. By \cite[Proposition 3.4.5]{GL}, $(m_!\mF)^\rW=0$ is equivalent to
\beq\label{vanishing}
 \oH_c^*(\bG_m,(m_!\mF)^\rW\otimes\mL')\is\oH_c^*(\bG_m,m_!\mF\otimes\mL')^\rW\is
 \oH_c^*(T,\mF\otimes m^*\mL')^\rW=0
 \eeq
 for all tame local system $\mL'$ on $\bG_m$. 
 Note that 
 we have $\rW_\chi\neq e$ if and only if 
 $\mL_\chi\is\mL'\boxtimes\mL'\is m^*\mL'$ for some tame local system
 $\mL'$ on $\bG_m$, thus~\eqref{vanishing} is equivalent to the condition that 
$\rW_\chi$ acts on
 $\oH_c^*(T,\mF\otimes\mL_\chi)$ via the sign character for any 
 tame character $\chi$, that is,
 $\mF$ is central. The claim follows.

We shall generalize the proof above for $\GL_2$ to $\GL_n$.
The argument 
involves mirabolic subgroups of $\GL_n$
as an essential ingredient.


\subsection{Mirabolic subgroups}\label{mirabolic}
We recall some geometric facts about Mirabolic subgroups, established in \cite{CN},
that will be used in the proof of Theorem \ref{main 2}.\footnote{
As mentioned in \emph{loc. cit.} the geometry of the conjugation action of Mirabolic has been described by Bernstein in \cite{B}.}

Let $V=\bbA^n$ be the standard $n$-dimensional vector space over 
$k$ with the standard basis $e_1,...,e_n$.
For any $1\leq m\leq n$
We define $F_m$ (resp. $E_m$) to be the subspace generated by $e_1,...,e_m$
(resp. $e_{m+1},...,e_n$).

Let $Q$ be the mirabolic subgroup of $G=\GL(V)$ consisting of 
$g\in G$ fixing the line generated by $v:=e_1$.
Let $U_Q$ be the unipotent radical of $Q$.
Consider the $Q$-conjugation equivariant stratification of $G$
\beq\label{X_m}
G=\bigsqcup_{m=1}^n X_m,
\eeq
where $X_m$ the subset of $G$ consisting of $g\in G$
such that the span of the vectors $v,gv,g^2v,...$ is of dimension $m$.

\begin{lemma}\label{det}
Consider the Chevalley quotient map $c:G\to \rW\backslash\backslash T\is\bbA^{n-1}\times\bG_m$, $c(x)=(a_1,...,a_n)$
where $t^n+a_1t^n+\cdot\cdot\cdot+a_n$ is the characteristic polynomial of $x\in G$
Let $x\in X_n$. Then the map 
$u\to c(ux)-c(x)$ induces a linear isomorphism
\[U_Q\is\bbA^{n-1}\times\{0\}\subset\bbA^{n-1}\times\bG_m\]
between $U_Q$ and the subspace of $\bbA^{n-1}\times\bG_m$
defined by $a_n=0$.
In particular, the map $c$ restricts to an isomorphism 
$U_Qx\is\bbA^{n-1}\times\{a_n\}\subset\bbA^{n-1}\times\bG_m$.

\end{lemma}
\begin{proof}
This is \cite[Proposition 4.1]{CN}. 
\end{proof}
\begin{remark}
The lemma above is a generalization of the isomorphism in~\eqref{key iso GL_2}
to the setting of mirabolic subgroups.
\end{remark}

\begin{lemma}\label{linear algebra}
Let $x\in X_m$. There exists an element $q\in Q$ such that 
$qxq^{-1}$ is of the form 
\beq\label{standard form}
\begin{bmatrix}
    x_{F_m} & y\\
    0 & x_{E_m} 
\end{bmatrix}
\eeq
where $x_{E_m}\in\GL(E_m)$, and 
$x_{F_m}\in\GL(F_m)$ has the form of a companion matrix 
\beq\label{companion}
x_{F_m}=\begin{bmatrix}
    0  & 0&  \dots  &0 & 0&-a_{m} \\
      1  &0&   \dots     & 0  & 0  &-a_{m-1} \\
      0  &1&   \dots      & 0 & 0  &-a_{m-2} \\
      \vdots &   \vdots   &   \ddots   & \vdots&\vdots&\vdots  \\
      0   &    0& \dots  &1   &0    & -a_{2}\\
      0  &    0& \dots  &0   &1    & -a_{1}
    \end{bmatrix}
\eeq

\end{lemma}
\begin{proof}
Straightforward exercise in linear algebra.
\end{proof}

Let $P_m$ be the parabolic group of $G$ consisting of 
$g$ such that $gF_m=F_m$. Let $L_m$ be the standard Levi subgroup of $P_m$ and
$U_{m}$
be the unipotent radical of $P_m$ consisting of matrices of the form
\[u_m=\begin{bmatrix}
    \Id_{F_m} & v_m\\
    0 & \Id_{E_m} 
\end{bmatrix}\]
where $u_m\in\Hom(E_m,F_m)$.
Let $U_1$ and $U_{m-1}$ be the subgroup of $U_{m}$ 
consisting of $u_m$ as above with $v_m\in\Hom(E_m,F_1)$
and $v_m\in\Hom(E_m,F_{m-1})$ respectively.
\begin{lemma}\label{key lemma}
Let $x_{F_m}\in\GL(F_m)$ be a linear map such that,
for any $1\leq j\leq m$,
$x_{F_m}(F_{j})\subset F_{j+1}$ and the induced map 
$F_j/F_{j-1}\to F_{j+1}/F_j$ is an isomorphism.
Let $x_{E_m}\in\GL(E_m)$ be an arbitrary linear map.
Then the action of $U_1\times U_{m-1}$ on the space of matrices of the form
\beq\label{standard form 2}
x=\begin{bmatrix}
    x_{F_m} & y\\
    0 & x_{E_m}
\end{bmatrix}
\eeq
given by $(u_1,u_{m-1})x=u_{m-1}u_1xu_{m-1}^{-1}$
is simply transitive.\footnote{Note that $u_{m-1}u_1xu_{m-1}^{-1}=u_1u_{m-1}xu_{m-1}^{-1}$.}

\end{lemma}
\begin{proof}
This is \cite[Lemma 3.2]{CN}.\footnote{There is minor mistake in the computation of $u_1u_mxu_m^{-1}$ in \emph{loc. cit.}: 
is should be $u_1u_mxu_m^{-1}=\begin{bmatrix}
    x_{F} & y+v_1x_E+v_{m-1}x_E-x_Fv_{m-1}\\
    0 & x_{E} 
\end{bmatrix}$.
The same proof goes through after this minor correction.
}
\end{proof}

Let $Q_{L_m}=L_m\cap Q$ be a mirabolic subgroup of $L_m$
and let $U_{Q_{L_m}}$ be the unipotent radical of 
$Q_{L_m}$ consisting of matrices 
\[\begin{bmatrix}
    1&v&0&\\
    0 & \Id_{m-1}&0\\
    0&0& \Id_{n-m} 
\end{bmatrix}\]
where $v=(v_1,...,v_{m-1})\in\bbA^{m-1}$ is a row vector.

\begin{lemma}\label{key lemma 2}
Let $x\in X_m$ be as in~\eqref{standard form}
with $x_{F_m}$ being as in~\eqref{companion}.
The morphism 
\[U_Q\times U_{m-1}\to U_mU_{Q_{L_{m}}}x\]
given by
$(u_Q,u_{m-1})\to u_{m-1}u_Qxu_{m-1}^{-1}$
is an isomorphism.\footnote{The isomorphism was used in \cite[p17]{CN}.}
\end{lemma}
\begin{proof}
Recall the subgroup $U_{1}$ of $U_m$ consisting of matrices of the form
\[\begin{bmatrix}
    \Id_{m}&v&\\
    0 & \Id_{m-n}&
 \end{bmatrix}\]
where $v\in\Hom(E_m,F_1)$.
Note that the multiplication map
\beq\label{factorization}
U_{1}\times U_{Q_{L_{m}}}\is U_Q,\ \ (u_1,u)\to u_1u
\eeq
is an isomorphism.  Note also that, for any $u\in U_{Q_{L_{m}}}$,
the $U_m$-orbit $U_mux$ through $ux$ consists of matrices of the form 
~\eqref{standard form 2} with $x_{F_m}$, $x_{E_m}$ fixed, and
$x_{F_m}$
satisfies the assumption 
of Lemma \ref{key lemma}.
Thus, by Lemma \ref{key lemma}, we have an isomorphism 
\[U_{1}\times U_{m-1}\to  U_mux\]
given by $(u_1,u_{m-1})\to u_{m-1}u_1uxu_{m-1}^{-1}$.
As $u$ varies over $U_{Q_{L_{m}}}$, we obtain the desired isomorphism 
\[U_Q\times U_{m-1}\stackrel{\eqref{factorization}}\is U_{1}\times U_{Q_{L_{m}}}\times U_{m-1}\to  U_mU_{Q_{L_{m}}}x,\ \ (u_Q=u_1u,u_{m-1})\to u_{m-1}u_1uxu_{m-1}^{-1}.\]

\end{proof}
\subsection{Proof of Theorem \ref{main 2}}
\quash{
\begin{lemma}\label{descent GL_n}
Assume $G=\GL_n$.
(1)
We have $\rW_\chi=\rW_\chi'$ for any tame character $\chi\in\calC(T)$.
(2) A complex $\mF\in\sD_\rW(T)$ is central (resp. $*$-central) if and only if 
$\calM_!(\mF\otimes\sign)$ (resp. $\calM_*(\mF\otimes\sign)$) descends to $\rW\backslash\backslash\calC(T)$.

\end{lemma}
\begin{proof}
The claim $\rW_\chi=\rW_\chi'$ can be checked directly and the second claim 
follows from Proposition \ref{characterization}.
\end{proof}
}
\begin{lemma}\label{reduction to perverse sheaf}
Assume $G=\GL_n$.
If Conjecture \ref{vanishing conj} is true for 
perverse central sheaves (resp. perverse $*$-central sheaves),
 then is it true for arbitrary central complexes (resp. arbitrary $*$-central complexes).\footnote{In fact, the Lemma remains true without the assumption $G=\GL_n$, see \cite[Lemma 7.5]{C}.}
\end{lemma}
\begin{proof}
By Remark \ref{*=!},
it is enough to prove the statement for $*$-central complexes.
By induction on the (finite) number of 
non vanishing perverse cohomology sheaves, we can assume the conjecture is true for $*$-central complexes 
in $^p\sD^{[a,b]}_\rW(T)$, $|a-b|\leq l$. 
Let $\mF\in{^p\sD}^{[a,b+1]}_\rW(T)$ be a $*$-central complex.
We need to show that $\pi_*\Phi_{\mF}$ is supported on 
$T=U\backslash B\subset U\backslash G$, where $\pi:G\to U\backslash G$
is the quotient map.
Consider the following distinguished triangle 
\[^p\tau_{\leq b}\mF\to\mF\to{^p\sH^{b+1}}(\mF)[-b-1]\to.\]
We claim that both $^p\tau_{\leq b}\mF,{^p\sH^{b+1}}(\mF)\in \sD_\rW(T)$
are $*$-central. 
Applying the functor $\Ind_{T\subset B}^G(-)^\rW$ to the above 
distinguished triangle, we obtain 
\[\Phi_{^p\tau_{\leq b}\mF}\to\Phi_\mF\to\Phi_{{^p\sH^{b+1}}(\mF)}[-b-1]\to.\]
By induction, both $\pi_*(\Phi_{^p\tau_{\leq b}\mF})$
and $\pi_*(\Phi_{{^p\sH^{b+1}}(\mF)})$ are supported on $T=U\backslash B\subset U\backslash G$
and it implies $\pi_*\Phi_\mF$ is also supported on $T$. 

It remains to prove the claim. 
Let $\chi_f\in\calC(T)_f$ and let $\calC(T)_{\ell,\chi_f}=\{\chi_f\}\times\calC(T)_\ell$
be the corresponding component.
Let $q:\calC(T)_{\ell,\chi_f}\to\rW_{\chi_f}\backslash\backslash\calC(T)_{\ell,\chi_f}$ be the quotient map.
Note that for $G=\GL_n$ we have $\rW_\chi=\rW_\chi'$ for all $\chi$ and hence
by Proposition \ref{characterization}, we have 
\[\calM_*(\mF\otimes\sign)|_{\calC(T)_{\ell,\chi_f}}=q^*\mG\] for a $\mG\in D_{\on{coh}}^b(\rW_{\chi_f}\backslash\backslash\calC(T)_{\ell,\chi_f})$. Since both $\calM_*$ and $q^*$ are exact, we have 
\[\calM_*(^p\tau_{\leq b}(\mF\otimes\sign))|_{\calC(T)_{\ell,\chi_f}}=q^*(\tau_{\leq b}\mG),\ \ \calM_*({^p\sH^{b+1}}(\mF\otimes\sign))|_{\calC(T)_{\ell,\chi_f}}=q^*\sH^{b+1}(\mG),\] 
and by Proposition \ref{characterization} again it implies 
$^p\tau_{\leq b}\mF$ and ${^p\sH^{b+1}}(\mF)$ are $*$-central.
The proof is complete.

\end{proof}

Consider 
$G=(\prod_{i\neq j}\GL_{n_i})\times\GL_{n_j}$ with 
maximal torus $T=(\prod_{i\neq j}T_{i})\times T_{j}$, and
Weyl grorup $\rW=(\prod_{i\neq j}\rW_i)\times\rW_j$. Here $\rW_k$ 
denotes the symmetric group of degree $n_k$.
Consider the following map 
\[\eta=\id\times\det:T=(\prod_{i\neq j}T_{i})\times T_{j}\longrightarrow T':=(\prod_{i\neq j}T_{i})\times\bG_m.\]
The symmetric group $\rW_j$ of degree $n_j$ acts 
on $T$ via 
the permutation action on the factor $T_j$ and 
$\eta$ is $\rW_j$-invariant.
Thus 
for any $\mF\in\sD_\rW(T)$ the push-forward 
$\eta_!(\mF)$ carries a natural $\rW_j$-action.
We have the following generalization of Deligne's result 
about $\rS_{\underline\lambda}$-action on hypergeometric sheaves 
$\Phi_{\underline\lambda}$ (see Proposition \ref{sign}) 
to central complexes. 

\quash{
Let $\eta:T\to T'$ be a morphism from the maximal torus of $G$ to another torus $T'$.
Let $\rW_\eta=\{w\in\rW|\eta\circ w=\eta\}$ be the stabilizer of $\eta$ in $\rW$. 
Then for any $\mF\in\sD_\rW(T)$ the push-forward 
$\eta_!(\mF)$ carries a natural $\rW_\eta$-action.
We have the following generalization of Deligne's result 
about $\rS_{\underline\lambda}$-action on hypergeometric sheaves 
$\Phi_{\underline\lambda}$ (see Proposition \ref{sign}) 
to central complexes. }

\begin{prop}\label{Deligne's lemma}
Let $\mF$ be a central complex in $\sD_\rW(T)$.
Then the above $\rW_j$-action on
$\eta_!(\mF)$ is given by the sign character 
$\sign_\rW:\rW_j\to\{\pm1\}$.
\end{prop}
\begin{proof}
It suffices to show that $\eta_!(\mF)^{\sigma=\id}=0$ for any involution
$\sigma\in\rW_j$.
By a result of Laumon \cite[Proposition 3.4.5]{GL}, it is enough to show that
for any tame local system
$\mL_\chi$ on $T'$ we have 
\[\oH_c^*(T',\eta_!(\mF)^{\sigma=\id}\otimes\mL_\chi)=0.\]
Note that, for any such
$\mL_\chi$, we have
 \[\oH_c^*(T',\eta_!(\mF)^{\sigma=\id}\otimes\mL_\chi)\is
 \oH_c^*(T',\eta_!(\mF)\otimes\mL_\chi)^{\sigma=\id}\is 
 \oH_c^*(T,\mF\otimes \eta^*\mL_\chi)^{\sigma=\id}.\]
Since $\eta^*\mL_\chi$ is a tame local system on $T$
 fixed by $\sigma$, the central property of $\mF$ implies that the action of 
 $\sigma$ on $\oH_c^*(T,\mF\otimes\eta^*\mL_\chi)$
 is trivial. Thus we have $\oH_c^*(T,\mF\otimes\eta^*\mL_\chi)^{\sigma=\id}=0$ and the isomorphism above implies 
$\oH_c^*(T',\eta_!(\mF)^{\sigma=\id}\otimes\mL_\chi)=0$
for all tame local system $\mL_\chi$. 
The proposition follows.
\end{proof}

The following proposition
generalizes \cite[Proposition 5.1]{CN} to 
central complexes:
\begin{proposition}\label{reduction 2}
Let $G$ be a direct product of general linear group and let $Q$
be a mirabolic subgroup of $G$ of the form
\[Q=(\prod_{i\neq j}\GL_{n_i})\times Q_{j}\subset G=\prod_i\GL_{n_i}\]
where $Q_j$ is the mirabolic subgroup of $\GL_{n_j}$.
Let $\mF\in\sD_\rW(T)^\heartsuit$ be a $\rW$-equivariant
central perverse sheaf on $T$.
Then for any $x\in G\setminus Q$, we have 
\[\oH^*_c(U_Qx,i^*\Phi_\mF)=0\]
where $i:U_Qx\to G$ is the inclusion map. 
\end{proposition}
\begin{proof}
If $n_j=1$, then $Q=G$ and the proposition holds vacuously.
Assume $n_j\geq 2$.
Consider the following stratification of $G$
\[G=\bigsqcup_{m=1}^{n_j} X_m=\bigsqcup_{m=1}^{n_j}\big(\prod_{i\neq j}\GL_{n_i}\times X_{j,m}\big)\]
where $X_{j,m}\subset\GL_{n_j}$ is the subset introduced in~\eqref{X_m}.
We have $Q=X_1$, thus the assumption $x\notin Q$ implies 
$x\in X_m$ for some $2\leq m\leq n_j$.

Consider the case $x\in X_{j,n_j}$. Since $X_{j,n_j}$ is contained in
$\GL_{n_j}^{\on{reg}}$, the Grothendieck-Springer simultaneous resolution implies a
Cartesian diagram
\[\xymatrix{\tilde X_{n_j}\ar[r]^{\tilde c\ \ \ \ \ \ \ \ \ \ }\ar[d]^{\tilde q}&(\prod_{i\neq j}\GL_{n_i})\times T_j\ar[d]^q\\
X_{n_j}\ar[r]^{c\ \ \ \ \ \ \ \ \ \ \ \ }&(\prod_{i\neq j}\GL_{n_i})\times\rW_j\backslash\backslash T_j}.\]
It follows that
\[\Phi_{\mF}|_{X_{n_j}}\is c^*q_!(\Phi_{L_j,\mF})^{\rW_j}\]
where $L_j=(\prod_{i\neq j}\GL_{n_i})\times T_j$ and 
$\Phi_{L_j,\mF}=\Ind_T^{L_j}(\mF)^{\rW_{L_j}}$.
Consider the map
\beq\label{det}
\on{det}_{\rW_j}=\Id\times\sigma:(\prod_{i\neq j}\GL_{n_i})\times\rW_j\backslash\backslash T_j\to
(\prod_{i\neq j}\GL_{n_i})\times\bG_m
\eeq
where $\sigma:\rW_j\backslash\backslash T_j\is\GL_{n_j}\backslash\backslash\GL_{n_j}\to\bG_m$
is given by the determinant function on $\GL_{n_j}$.
By Lemma \ref{det}, the restriction of $c$ to $U_Qx$ induces an isomorphism between $U_Qx$ and the fiber of~
~\eqref{det}
over the image $\on{det}_{\rW_j}\circ c(x)$ of $x$. 
Thus, to prove the desired vanishing, it is enough to show that 
\beq\label{vanishing over regular loci}
\on{det}_{j,!}(\Phi_{L_j,\mF})^{\rW_j}\is\on{det}_{\rW_j,!}q_!(\Phi_{L_j,\mF})^{\rW_j}=0
\eeq
where
\[\on{det}_j=\on{det}_{\rW_j}\circ q:L_j=(\prod_{i\neq j}\GL_{n_i})\times T_j\to
(\prod_{i\neq j}\GL_{n_i})\times\bG_m\]
is the determinant map on the $j$th component.
Consider the following Cartesian diagrams
\[\xymatrix{L_j=(\prod_{i\neq j}\GL_{n_i})\times T_j\ar[d]^{\on{det}_j=\Id\times\det}&(\prod_{i\neq j}\widetilde\GL_{n_i})\times T_j\ar[l]_{}\ar[r]\ar[d]^{\id\times\det}&T=(\prod_{i\neq j}T_{i})\times T_j\ar[d]^{\id\times\det}\\
L_j'=(\prod_{i\neq j}\GL_{n_i})\times \bG_m&\prod_{i\neq j}\widetilde\GL_{n_i}\times \bG_m\ar[l]\ar[r]&
T'=(\prod_{i\neq j}T_{i})\times \bG_m},\]
where the horizontal arrows are the map induced by 
the Grothendieck-Springer simultaneous resolution (see \eqref{G-S resolution}).
It follows that 
\beq\label{L_j'}
\on{det}_{j,!}\Phi_{L_j,\mF}\is\on{det}_{j,!}\Ind_{T}^{L_j}(\mF)^{\rW_{L_j}=\prod_{i\neq j}\rW_i}\is\Ind_{T'}^{L_j'}(\mF')^{\prod_{i\neq j}\rW_i}
\eeq
where 
\[\mF'=(\Id\times\det)_!(\mF).\]
By Proposition \ref{Deligne's lemma}, the action of $\rW_j$ on 
$\mF'$ is given by the sign character $\sign_\rW:\rW_j\to\{\pm1\}$.
Therefore, by \eqref{L_j'}, the action $\rW_j$ on 
$\on{det}_{j,!}\Phi_{L_j,\mF}$ is by the sign character and we have 
\[\on{det}_{j,!}(\Phi_{L_j,\mF})^{\rW_j}=0\]
as for $n_j\geq 2$ the sign character of $\rW_j$ is non-trivial.
This concludes the case when $x\in X_{n_j}$.

Consider the general case $x\in X_m$ with $2\leq m\leq n_j$.
By Lemma \ref{linear algebra}, we can assume the 
$j$th component $x_j\in\GL_{n_j}$ of $x$ is of the form 
form~\eqref{standard form} with $x_{j,F}$ being a companion form in~\eqref{companion}.
Let $P=\prod_{i\neq j}\GL_{n_i}\times P_m$ denote the parabolic subgroup
of $G$ where $P_m$ is the parabolic subgroup of $\GL_{n_j}$ introduced in
Section \ref{mirabolic}. Let $L=(\prod_{i\neq j}\GL_{n_i})\times L_m$ and $U_P=(\prod_{i\neq j}\Id_{n_i})\times U_m$
be the standard Levi subgroup of $P$ and its unipotent radical. 
We have $x\in P$ and let $x_L$ be its image in $L$.
Consider the mirabolic subgroup $Q_L$ of $L$ with 
unipotent radical $U_{Q_L}=(\prod_{i\neq j}\Id_{n_i})\times Q_{L_m}$.
By applying the result obtained above to the case to $L$, we obtain 
\[\oH_c^*(U_{Q_L}x_L,\Phi_{L,\mF}|_{U_{Q_L}x_L})=0,\]
where $\Phi_{L,\mF}=\Ind_{T}^L(\mF)^{\rW_L}$.
Note that, by Proposition \ref{Mackey formula}, we have 
$\Phi_{L,\mF}\is\Res_{L\subset P}^G\Phi_{\mF}$, and above cohomology vanishing implies 
\beq\label{reduction to L}
\oH_c^*(U_PU_{Q_L}x,\Phi_{\mF}|_{U_PU_{Q_L}x})=0.
\eeq
Using Lemma \ref{key lemma 2}, we see that the map
\[U_Q\times\big(\prod_{i\neq j}\Id_{n_i}\times U_{m-1}\big)\to U_PU_{Q_L}x\]
given by $(u_q,u_{m-1})\to u_{m-1}u_qxu_{m-1}^{-1}$
is an isomorphism. Since $\Phi_{\mF}$ is $G$-conjugation equivariant, 
\eqref{reduction to L} implies
\[\oH_c^*(U_Qx,\Phi_{\mF}|_{U_Qx})=0.\]
The proposition follows.
\end{proof}

We are ready to prove Theorem \ref{main 2}.
Let $G=\GL_n$ and let $Q$ be the mirabolic subgroup.
Let $\mF\in\sD_\rW(T)$ be a central complex.
We would like to show that 
$\oH_c^*(Ux,\Phi_\mF|_{Ux})=0$
for $x\in G\setminus B$. 
By Lemma \ref{reduction to perverse sheaf}, we can assume $\mF$ is a perverse sheaf.
Consider the case $x\notin Q$. Then $ux\notin Q$ for all $u\in U$ and, by Proposition \ref{reduction 2}, we have 
$\oH_c^*(U_Qux,\Phi_{\mF}|_{U_Qux})=0$. 
The desired vanishing $\oH_c^*(Ux,\Phi_{\mF}|_{Ux})=0$
follows from the Leray spectral sequence associated to the 
map $Ux\to U_Q\backslash Ux$.
Now assume $x\in Q$. 
Let $x_L$ be the image of $x$ in the standard Levi $L$ of Q.
Note that we have $x_L\notin B_L=B\cap L$ and 
\[\oH_c^*(Ux,\Phi_{\mF}|_{Ux})=\oH_c^*(U_{B_L}x_L,\Res_{L}^G\Phi_\mF|_{U_{B_L}x_L})=
\oH_c^*(U_{B_L}x_L,\Phi_{L,\mF}|_{U_{B_L}x_L})\]
where $U_{B_L}$ is the unipotent radical of $B_L$.
Now, using Proposition \ref{reduction 2}, we can conclude by an induction 
argument.

\end{document}